\documentclass[12pt]{article}
\usepackage{amsmath,amssymb}
\usepackage{amsthm}

\def\MC{\mathcal{M}}

\def\LC{\mathcal{L}}
\def\TC{\mathcal{T}}

\def\C{\mathbf{C}}
\def\E{\mathbf{E}}

\def\P{\mathbf{P}}
\def\R{\mathbf{R}}

\def\1{\mathbf{1}}

\def\tr{\rm{tr}}

\def\al{\alpha}
\def\be{\beta}
\def\pa{\partial}

\def\de{\delta}
\def\ga{\gamma}

\newtheorem{prop}{Proposition}[section]
\newtheorem{theorem}{Theorem}[section]
\newtheorem{lemma}{Lemma}[section]
\newtheorem{exer}{Exercise}[section]

\newtheorem{remark}{Remark}

\newcommand{\la}{\lambda}
\newcommand{\si}{\sigma}

\newcommand{\om}{\omega}

\newcommand{\Om}{\Omega}

\begin{document}
\title{Quantum games: a survey for mathematicians}

\author {
 {\it Vassili N. Kolokoltsov} \\ Department of Statistics, University of  Warwick,
\\
email: v.kolokoltsov@warwick.ac.uk
\\
 }
\date{This is a draft of the final chapter 13 of the Second Edition of the book V. Kolokoltsov and O. Malafeyev
"Understanding game theory", World Scientific, 2019}

\maketitle

\begin{abstract}
 Main papers on
quantum games are written by physicists for physicists, and the inevitable exploitation
of physics jargon may create difficulties for mathematicians or economists. Our
goal here is to make clear the physical content and to stress the new features
of the games that may be revealed in their quantum versions.
Some basic knowledge of quantum mechanics is a necessary prerequisite
for studying quantum games. The most fundamental facts are collected in
Section \ref{secpreliomfinitequ} describing closed finite-dimensional systems
and complemented by the rules of quantum measurements in Section \ref{secquantmes}.
These facts are sufficient for the main trend of our introductory exposition.
 However, for further developments of quantum games
one needs some notions of open quantum systems. They are presented in
Section \ref{secOpenSyst}, to which we refer occasionally, and which supplies the background
that is necessary for reading
modern research papers. The main sections \ref{secMeyGame}
-\ref{secVaronEWL} build the foundations of quantum games via the basic examples. We omit sometimes the
lengthy calculations (referring to the original papers) once the physical part
is sorted out and the problem is reformulated as pure game-theoretic problem
of calculating the Nash or dominated equilibria. Section \ref{secQGcontstate}
is devoted to quantum games arising from the classical games with infinite
state space (like the classical Cournot duopoly). Section \ref{secgenthequaga}
touches upon general theory of finite quantum static games. Further links
and references are provided in Section \ref{secfinconq}.
\end{abstract}

\section{On finite-dimensional quantum mechanics}
\label{secpreliomfinitequ}

Finite-dimensional quantum systems are described by
finite-dimensional complex Euclidean or Hilbert spaces $H=\C^n$,
 equipped with the standard scalar product $(.,.)$,
which is usually considered to be {\it anti-linear}\index{anti-linear map}
or {\it conjugate linear} with respect to the first argument,
$(a u,v)=\bar a (u,v)$ for $a \in \C$, and linear with respect to the second one:
$(v, a u)=a (v,u)$. For any orthonormal basis $e_1, \cdots, e_n$ in $H$,
the scalar product writes down in coordinates as $(u,v)=\sum_j \bar u_j v_j$,
where $u=\sum_j u_j e_j$, $v=\sum_j v_j e_j$.
The usual Euclidean norm is $\|x\|=\sqrt{(x,x)}$.
 These systems are referred to as {\it qubits}\index{qubits},
{\it qutrits}\index{qutrits} in case of $n=2$ or $3$,
and {\it qunits}\index{qunits} for general $n$ (or rather {\it qudits} with general $d$).

{\it Pure states}\index{pure states of a quantum system} of a quantum system
described by such space $H$
are unit vectors in $H$.

\begin{remark} More precisely, two vectors that differ by a multiplier
are considered to describe the same state, so that the state space is the projective space
$\C P^{k-1}$ of the equivalence classes of vectors with equivalence defined as proportionality.
\end{remark}

The space of all $n\times n$ matrices $A=(A_{ij})$ can be considered
as the space of all linear operators $\LC(\C^n)$ in $\C^n$. This correspondence
is described equivalently either via the action on the basis vectors as $Ae_j=\sum_k A_{kj}e_k$,
or in coordinates as $(Au)_j=\sum_k A_{jk}u_k$. $\LC(\C^n)$ is a space
of dimension $n^2$ with the usual operator norm defined as
\begin{equation}
\label{defopernorm}
\|A\|=\sup_{\|x\|=1}\|Ax\|.
\end{equation}

Similarly the space $\LC(H_1,H_2)$, $H_2=\C^m$, of linear operators $H_1\to H_2$
can be identified with the space of $m\times n$ matrices. It becomes a norm space
under the standard operator norm \eqref{defopernorm},
where $\|x\|$ is the norm in $H_1$ and $\|Ax\|$ is the norm in $H_2$.

The subspace of $\LC(\C^n)$ consisting of self-adjoint or Hermitian matrices,
defined by the equation $A^*=A$, where $A^*=\bar A^T$ ($T$ for transpose and
the bar for complex conjugation) will be denoted $\LC_s(\C^n)$. The trace of
$A\in \LC(\C^n)$ is defined as ${\tr} \, A=\sum_j A_{jj}$. It implies that
 ${\tr} (AB)= {\tr} (BA)$, which in turn  implies that ${\tr} (C^{-1}AC)={\tr} A$
 for any invertible $C$. Consequently one obtains  two other equivalent expressions
 for the trace of  $A\in \LC_s(\C^n)$: ${\tr} \, A=\sum_j \la_j$, where
 $\{\la_j\}$ is the collection of all eigenvalues of $A$ and
\[
{\tr} \, A=\sum_j (e_j, Ae_j),
\]
where $\{e_j\}$ is an arbitrary orthonormal basis in $\C^n$.

The space $\LC(\C^n)$ is a Hilbert space with respect to the scalar product
  \begin{equation}
\label{scalprodoper}
(A,B)_{HS}={\tr} (A^*B)=\sum_{i,j}\bar A_{ij}B_{ij},
\end{equation}
 with the corresponding norm called the {\it Hilbert-Schmidt norm}\index{Hilbert-Schmidt norm}
\[
\|A\|_{HS}=[{\tr} (A^*A)]^{1/2}.
\]
For Hermitian operators it simplifies to ${\tr} (A^*B)={\tr} (AB)$.

Any  $A\in \LC_s(\C^n)$ is diagonizable, meaning that there exists a unitary
$U$ such that $A=U^*DU$, where $D=D(\la_1, \cdots, \la_n)$ is diagonal with the
eigenvalues $\{\la_j\}$ of $A$ on the diagonal. Then $|A|$ is defined as the
positive operator $|A|=U^*|D|U$ with $|D|=D(|\la_1|, \cdots, |\la_n|)$,
the diagonal operator with the numbers $\{|\la_j|\}$ on the diagonal.
The functional $A\mapsto {\tr} (|A|)$ defines yet another norm on $\LC_s(\C^n)$,
the {\it trace norm}\index{trace norm}:
\[
\|A\|_{tr}= {\tr} (|A|)=\sum_j |\la_j|.
\]

The key point is that this norm is dual to the usual operator norm with respect
to the duality provided by the trace: for $A\in \LC_s(\C^n)$,
  \begin{equation}
\label{dualopernorm}
\|A\|_{tr} =\sup_{\|B\|=1} |{\tr} (AB)|, \quad \|A\| =\sup_{\|B\|_{tr}=1} |{\tr} (AB)|.
\end{equation}

To show these equations it is handy to work in the basis where $A$ is diagonal:
$A=D(\la_1, \cdots, \la_n)$. Then ${\tr}(AB)=\sum_j\la_j B_{jj}$.
Choosing $B$ to be diagonal with $B_{ij}$ equal the sign of $\la_j$ it follows that
\[
\sup_{\|B\|=1} |{\tr} (AB)|\ge \sum_j |\la_j|=\|A\|_{tr}.
\]
On the other hand,
\[
\sup_{\|B\|=1} |{\tr} (AB)|=
\sup_{\|B\|=1} |\sum_j\la_j B_{jj}|
\le  \sum_j |\la_j| \max_j|B_{jj}|
\le  \sum_j |\la_j|.
\]
Therefore the first equation in \eqref{dualopernorm} is proved. The second equation is proved similarly.

For two spaces $H_1=\C^n$ and $H_2=\C^m$ with orthonormal bases $e_1, \cdots, e_n$
and $f_1, \cdots f_m$ the {\it tensor product space}
$H_1 \otimes H_2$ can be defined as the $nm$-dimensional space generated by vectors denoted
$e_j\otimes f_k$, so that any $\psi \in H_1\otimes H_2$ can be represented
as
  \begin{equation}
\label{eqtensprodbas}
 \psi=\sum_{j=1}^n \sum_{k=1}^m a_{jk}  e_j \otimes f_k.
\end{equation}
 The {\it tensor product}\index{tensor product} of any two-vectors $u=\sum_j u_j e_j$,
$v=\sum_k v_k f_k$ is defined as the vector
\[
 u\otimes v=\sum_{j,k} u_jv_k e_j\otimes f_k.
\]

Similarly, the tensor product is defined for several Hilbert spaces. Namely,
the product $H=H_1\otimes H_2 \otimes \cdots \otimes H_K$ of spaces of dimensions $n_1 \cdots , n_K$
 can be described as the $n=\prod n_j$-dimensional space generated by vectors denoted
$e^1_{j_1}\otimes \cdots \otimes e^K_{j_K}$ and called the tensor products of
$e^1_{j_1}, \cdots , e^K_{j_K}$, where $\{e^m_1, \cdots , e^m_{n_m}\}$ is an orthonormal basis in $H_m$.

Any $\psi \in H_1\otimes H_2$ can be represented by the so-called
{\it Schmidt decomposition}\index{Schmidt decomposition}
  \begin{equation}
\label{eqSchmidtDecomp}
\psi=\sum_{j=1}^{\min(n,m)} \la_j \xi_j \otimes \eta_j,
\end{equation}
where $\la_j\ge 0$, and $\{\xi_j\}$ and $\{\eta_j\}$ are some orthonormal bases in $H_1$ and $H_2$.
In fact, one just has to write the singular decomposition of the matrix $A=(a_{jk})$ from \eqref{eqtensprodbas},
namely to represent it as $A=UDV^T$, where $D=D(\la_1, \la_2, \cdots)$ is diagonal
of dimension $m\times n$ and $U,V$ are unitary matrices. Here $\la_j^2$ are the (common)
eigenvalues of the matrices $A^*A$ and $AA^*$ with $\la_j\ge 0$. Then the vectors $\xi_j=U_{lj}e_l$ and
$\eta_k=V_{lj}f_l$ form orthonormal bases in $H_1$ and $H_2$ and
\[
\psi=\sum_{j,k,l} U_{jl}\la_l V_{kl}  e_j \otimes f_k
\]
equals \eqref{eqSchmidtDecomp}.

Pure states of the tensor product $\psi \in H_1\otimes H_2$ are called
{\it entangled}\index{entangled states} (the term introduced by E. Schr\"odinger in 1935)
if they cannot be written in the product form
$\psi=u\otimes v$ with some $u,v$. It is seen that $\psi$ is not entangled if and only if
its Schmidt decomposition has only one nonzero term. On the other hand, a pure state
$\psi \in H_1\otimes H_2$ is called {\it maximally entangled}\index{entangled states!maximally},
if its  Schmidt decomposition has the maximal
number of nonvanishing $\la_j$ and they all are equal, and thus they equal $1/\sqrt{\min(n,m)}$.

In physics one usually works in {\it Dirac's notations}\index{Dirac's notations}.
In these notations usual vectors $u=\sum_j u_j e_j \in H$ are referred
 to as {\it ket-vectors}, are denoted $|u\rangle$ and are considered to be
 column vectors with coordinates $u_j$. The corresponding {\it bra-vectors}
are denoted $\langle u|$ and are considered to be row vectors with coordinates
 $\bar u_j$. These notations are convenient, because they allow to represent
both scalar and tensor products as usual matrix multiplications: for two
ket-vectors $|u\rangle$ and $|v\rangle$ we have
\[
\langle u| v\rangle=\langle u|.| v\rangle =(u,v)=\sum_j \bar u_j v_j,
\]
\[
| v\rangle \langle u|=v\otimes \bar u =\sum_{j,k}  v_j\bar u_k e_j\otimes e_k.
\]
Therefore the latter product is often identified with the $n\times n$-matrix $\rho$
with the entries $ \rho_{jk}=v_j\bar u_k$.
As matrices, they act on vectors $w=|w\rangle$ in the natural way:
\[
\rho w=| v\rangle \langle u| w\rangle,
\quad (\rho w)_j= \sum_k \rho_{jk} w_k =v_j\sum_k \bar u_k w_k=v_j\langle u| w\rangle.
\]
On the other hand, the bra-vectors form the space $H^*$
 of linear functionals on $H$ specified via the scalar product.

\begin{remark} It is worth stressing that the operation of conjugation (usually denoted
by bar or a star) in a Hilbert space $H$ (defined as an anti-linear convolution map $A:H\to H$,
the latter meaning that $A^2=\1$) is not unique and depends on a chosen 'real' basis.
For instance, in $\C=\R^2$, a reflection with respect to any real line in $\C$
(chosen to be real in $\C$), or analytically, any map
of type $v\to \bar v e^{i\phi}$ with any real $\phi$  defines such a convolution.
\end{remark}

Continuing the analogy we see that the tensor product $H_1\otimes H_2^*$ is naturally isomorphic to
the space of linear operators  $\LC(H_2,H_1)$. Namely, for orthnormal bases $\{e_i\}, \{f_j\}$ in
$H_1$ and $H_2$ any
\[
X=\sum_{i,j} X_{ij} e_i \otimes \bar f_j\in H_1 \otimes H_2^*
\]
can be identified with the operator
\begin{equation}
\label{tensorprodcoor}
X=\sum_{ij} X_{ij}|e_i\rangle \langle f_j|: f_k \mapsto \sum_i X_{ik} |e_i\rangle
\end{equation}
with the matrix $X_{ij}$.
These matrix elements can be written in two equivalent forms:
\begin{equation}
\label{tensorprodcoor2}
X_{ij}= \langle e_i|X f_j\rangle={\tr}(X \, |f_j\rangle \langle e_i| ).
\end{equation}

{\it General or mixed states}\index{mixed states of a quantum system}, also
referred to as {\it density matrices}\index{density matrix} or {\it density operators}
of a qunit are defined to be non-negative $n\times n$-matrices $\rho$ with
 unit trace: ${\tr} \, \rho=1$. The {\it state space of a qunit} is usually defined either as
 the set of all density matrices (which is not a linear space) or the linear space
 generated by this set, that is, the space
 of all self-adjoint (or Hermitian) matrices equipped with the trace-norm and thus
 denoted $\TC_s(\C^n)$. The cones of positive elements of $\TC_s(\C^n)$ or $\LC_s(\C^n)$
 are denoted $\TC^+(\C^n)$ and $\LC^+(\C^n)$.

Any pure state $|u\rangle$ defines the density matrix $\rho= | u\rangle \langle u|$,
which is a one-dimensional projector. Thus pure states are naturally inserted in the
 set of all states. Moreover, if $\{e_j\}$ is an orthonormal basis in $H$, then the
matrices $|e_i\rangle\langle e_j|$ of rank $1$ form an orthonormal basis in $\LC_s(H)$
with respect to the scalar product   \eqref{scalprodoper}.
The quantitative deviation of a state from being pure is usually assessed either via the
{\it entropy}\index{entropy} of a state $S(\rho)=-{\tr} [\rho \ln (\rho)]$ (in a basis where
$\rho=D(\la_1, \cdots, \la_n)$ is diagonal, $S(\rho)=-\sum_j \la_j \ln \la_j$) or the
{\it purity}\index{purity} of a state, $P(\rho)={\tr} [\rho^2]$, because, as is seen directly,
$S(\rho)=0$ (respectively $P(\rho)=1$) if and only if $\rho$ is pure.

Mixed states $\rho$ in $H_1\otimes H_2$ are called {\it separable}
\index{mixed states of a quantum system!separable}, if they can be represented as
\[
\rho=\sum_j p_j \rho_j^1 \otimes \rho^2_j
\]
with some finite collection of states $\rho_j^1$ and $\rho_j^2$, and some $p_j>0$.
 Otherwise they are called
{\it entangled}.\index{mixed states of a quantum system!entangled}

\begin{exer} If $\rho \in H_1\otimes H_2$ is pure, $\rho=|\psi\rangle \langle \psi|$, then it is
separable if and only if $\psi=u\otimes v$ with some $u$, $v$. Equivalently, pure
  $\rho=|\psi\rangle \langle \psi|$ is entangled  if and only if $\psi$ is entangled.
\end{exer}

Possible {\it transformations of closed quantum systems} are assumed to be
given by {\it unitary matrices} $U$: $UU^*=U^*U=\1$. They act on vectors as
usual left multiplication: $w\mapsto Uw$, or in Dirac's notation
$|w\rangle \mapsto U |w\rangle=|Uw\rangle$, and on the density matrices by the
"dressing":
\begin{equation}
\label{eqdressingdef}
\rho \mapsto U \rho U^*.
\end{equation}
These actions are consistent
with the identification of vectors and pure states, since
\[
[U(\phi \otimes \bar \psi)U^*]_{ij}=\sum_{kl}U_{ik}\phi_k\bar \psi_l U^*_{lj}
\]
\[
=\sum_{kl}U_{ik}\phi_k\bar \psi_l \bar U_{jl}=(U\phi)_i (\overline{U\psi})_j
=(U\phi \otimes \overline{U\psi})_{ij}.
\]
The group of unitary matrices in $\C^n$ is denoted $U(n)$ and its
subgroup consisting of matrices with the  unit determinant is denoted $SU(n)$.

Of particular importance is the qubit arising from two-dimensional space, with the basis
\[
e_0=\left(\begin{aligned}
& 1 \\
& 0
\end{aligned}
\right)=|0\rangle, \quad e_1=\left(\begin{aligned}
& 0 \\
& 1
\end{aligned}
\right)=|1\rangle.
\]
As seen by direct inspection, the state space $\TC_s(\C^2)$ of the qubit is $4$- dimensional real space,
the most convenient basis given by the unity matrix $\1=I$ (we shall use both notations)
and the three {\it Pauli matrices}\index{Pauli matrices} (we show all three standard notations),
\[
\si_1=\si_x=X=\left(\begin{aligned}
& 0 \quad 1 \\
& 1 \quad 0
\end{aligned}
\right),
\quad
 \si_2=\si_2=Y=\left(\begin{aligned}
& 0 \quad -i \\
& i \quad \quad 0
\end{aligned}
\right),
\quad
\si_3=\si_z=Z=\left(\begin{aligned}
& 1 \quad \quad 0 \\
& 0 \quad -1
\end{aligned}
\right).
\]
Any density matrix of a qubit can be written uniquely in the form
\begin{equation}
\label{eqBlochspherestate}
\rho= \frac12 \left(\begin{aligned}
& 1+x_3 \quad x_1-x_2i \\
& x_1+x_2i \quad 1-x_3
\end{aligned}
\right)=\frac12(I+x_1\si_1+x_2 \si_2 +x_3 \si_3)
\end{equation}
with real $x_1,x_2,x_3$ satisfying $x_1^2+x_2^2+x_3^2\le 1$ (this is seen to be the
condition of positivity). These $x_1,x_2,x_3$ are called the {\it Stokes parameters}
\index{Stokes parameters} of a density matrix. Thus qubit is topologically a unit ball,
often referred to as the {\it Bloch sphere}\index{Bloch sphere}. Pure states are
distinguished by the property $\det \rho=0$, so that the pure states are characterized
by the condition $x_1^2+x_2^2+x_3^2= 1$ and thus form a two-dimensional sphere.

The group $SU(2)$ has many useful representations (regularly used in physics) that we
shall describe now. Direct inspection shows that any $U\in SU(2)$ has the form
  \begin{equation}
\label{eqgenunit2dim1}
U= \left(\begin{aligned}
& \quad u  \quad  v \\
& -\bar v  \quad \bar u
\end{aligned}
\right), \quad |u|^2+|v|^2=1.
\end{equation}
Writing $|u|=\cos\theta, |v|=\sin \theta$, for $\theta \in [0,\pi/2]$ we can represent
$u=e^{i\phi} \cos \theta$, $v=e^{i\psi} \cos \theta$ with some $\phi, \psi \in [-\pi, \pi]$
and thus any $U$ in $SU(2)$ as
  \begin{equation}
\label{eqgenunit2dim2}
U= \left(\begin{aligned}
& \quad e^{i\phi} \cos \theta  \quad  e^{i\psi} \sin \theta \\
& -e^{-i\psi} \sin \theta  \quad e^{-i\phi} \cos \theta
\end{aligned}
\right).
\end{equation}

Changing $\psi$ to $\psi+\pi/2$ it can be also equivalently written as
   \begin{equation}
\label{eqgenunit2dim3}
U= \left(\begin{aligned}
& \,\, e^{i\phi} \cos \theta  \quad  \,\, ie^{i\psi} \sin \theta \\
& ie^{-i\psi} \sin \theta  \quad  e^{-i\phi} \cos \theta
\end{aligned}
\right).
\end{equation}

Since for any operator $A$ such that $A^2=\1$, we have
\[
\exp\{ixA\}=\cos x \1 +i\sin x \, A,
\]
for any real $x$, it follows that
  \begin{equation}
\label{eqgenunit2dim3a}
e^{ibZ}=\left(\begin{aligned}
&  e^{ia} \quad \,\, 0 \\
& 0 \quad e^{-ia}
\end{aligned}
\right),
\quad
e^{icY}=\left(\begin{aligned}
& \quad \cos c \quad \sin c  \\
& -\sin c \quad \cos c
\end{aligned}
\right),
\end{equation}
so that
   \begin{equation}
\label{eqgenunit2dim4}
 e^{ibZ}e^{icY}e^{idZ}=
 \left(\begin{aligned}
& \quad e^{i(b+d)} \cos c  \quad  \,\, e^{i(b-d)} \sin c \\
& -e^{-i(b-d)} \sin c  \quad  e^{-i(b-d)} \cos c
\end{aligned}
\right).
\end{equation}
Comparing with   \eqref{eqgenunit2dim3} we see that \eqref{eqgenunit2dim4} is yet another way
to represent arbitrary element of $SU(2)$. This way is referred to as the $Z-Y$ decomposition or
the {\it Cartan decomposition} of the elements of $SU(2)$.

Finally, \eqref{eqgenunit2dim2} can be rewritten as
   \begin{equation}
\label{eqgenunit2dim5}
 U=\cos \phi \cos \theta \1 +i \sin \psi \sin \theta \si_x
 +i\cos \psi \sin \theta \si_y +i \sin \phi \cos \theta \si_z,
\end{equation}
and thus
   \begin{equation}
\label{eqgenunit2dim5a}
 U=u_0 \1 +i u_1 \si_1
 +i u_2 \si_2 +i u_3 \si_3,
\end{equation}
with real $u_k$ satisfying the condition $u_0^2+u_1^2+u_2^2+u_3^2=1$.
this representation shows that topologically $SU(2)$ is a unit sphere in $\R^4$.

 Any element $A$ of $\TC_s(\C^2)$ can be written in unique way as
\begin{equation}
\label{eqrepqubitstatespace}
A=x_0 I +x_1\si_1 +x_2\si_2 +x_3\si_3
\end{equation}
with real $x_j$.

Let us see how $u\in SU(2)$ act on $\TC_s(\C^2)$ via the dressing $A\mapsto UAU^*$:
$e^{i\phi \si_3}$ leaves $\si_3$ invariant and acts on $\si_1,\si_2$ as rotation with the matrix
\[
\left(\begin{aligned}
& \quad  \cos (2\phi)  \quad  - \sin (2\phi) \\
& -\sin (2\phi)  \quad  \quad \cos (2\phi)
\end{aligned}
\right);
\]
and similarly,
$e^{i\psi \si_2}$ leaves $\si_2$ invariant and acts on $\si_1,\si_3$ as rotation with the matrix
\[
\left(\begin{aligned}
& \quad  \cos (2\psi)  \quad  - \sin (2\psi) \\
& -\sin (2\psi)  \quad  \quad \cos (2\psi)
\end{aligned}
\right).
\]
But these rotations generate the group $SO(3)$ in $\R^3$ with the basis $\si_1, \si_2, \si_3$. Hence
any rotation in this space can be achieved via dressing with certain $u\in SU(2)$.

How to represent elements of $O(3)\setminus SO(3)$ via dressing? Here one needs anti-unitary operators.

A mapping $A:H\to H$ is called {\it ani-linear}\index{anti-linear map}
 or {\it conjugate-linear} if $A(v+w)=Av+Aw$ for $v,w\in H$ and
$A(a v)=\bar a v$, $a\in \C$. The simplest example is the equivalence between bra and ket vectors:
$|x\rangle \mapsto \langle x|$, or equivalently, just the mapping $\psi\to \bar \psi$.

To any linear map $A:H\to H$ there corresponds an anti-linear map
$\tilde A: v\mapsto A \circ s (v)$, where $s$ is the conjugation: $s(v)=\bar v$.
 If $A$, $B$ and $\rho$ are linear operators, then
$\tilde A \rho \tilde B$ is also a linear operator given by the matrix $A\bar \rho \bar B$.

An anti-linear map $A$ is called {\it anti-unitary}\index{anti-unitary maps}
 if $(Av,Aw)=(w,v)=(\bar v, \bar w)$. It is seen that $U$ is unitary if and only if
$\tilde U=U\circ s$ is anti-unitary, and for any two unitary operators $U,V$ the operators
$\tilde U \circ V$ and $U\circ \tilde V$ are anti-unitary. The famous {\it Wigner theorem}
\index{Wigner theorem} states that any mapping $U$ in a Hilbert space of dimension $d>2$
such that $|(Uu,Uv)|=|(u,v)|$ for all vectors $u,v$ is either unitary or anti-unitary.

\begin{exer} As an elementary version of the Wigner theorem show that any mapping
$M:\R^n\to \R^n$, which either preserves the scalar product or is continuous and preserves
the magnitude of the scalar product, is necessarily a linear orthogonal operator.
\end{exer}

If $\tilde U=U\circ s$ is an anti-unitary operator in $\C^2$, then the mapping
\begin{equation}
\label{eqantiunitdress}
\rho \mapsto \tilde U\rho \tilde U^{-1}=U \bar \rho \overline{U^{-1}} =U \bar \rho U^T
\end{equation}
acting on matrices \eqref{eqrepqubitstatespace} preserves $x_0$ and transforms
$x=(x_1,x_2,x_3)$ to $Ox$ with some orthogonal operator $O$ with $\det O=-1$;
and vice versa, any such transformation $O$ can be obtained
in this way. To prove this claim it is sufficient to show that
such representation is available for the reflection $R:(x_1,x_2,x_3)\mapsto (x_1, -x_2,x_3)$,
because  any orthogonal transformation $O$ with $\det O=-1$ can be written
as $O=R \circ S$ with $S\in SO(3)$. The reflection $R$ can be obtained from the anti-unitary
operator $\tilde U$ with $U= e^{ibZ}$ from \eqref{eqgenunit2dim3a}. In fact, we see by
\eqref{eqantiunitdress} that
\[
\tilde U \rho \tilde U^{-1}= \left(\begin{aligned}
&  e^{ia} \quad \,\, 0 \\
& 0 \quad e^{-ia}
\end{aligned}
\right) \bar \rho \left(\begin{aligned}
&  e^{ia} \quad \,\, 0 \\
& 0 \quad e^{-ia}
\end{aligned}
\right)
\]
and hence
\[
\tilde U \si_1 \tilde U^{-1}=\si_1,
\quad \tilde U \si_2 \tilde U^{-1}=-\si_2,
\quad \tilde U \si_3 \tilde U^{-1}=\si_3.
\]

For two qubits one has a straightforward criterion for states to be entangled. In fact, if
$\psi=\psi_0 |0\rangle+\psi_1 |1\rangle$, $\phi=\phi_0 |0\rangle+\phi_1 |1\rangle$, then
\begin{equation}
\label{eqproducttwoqubitsv}
\psi \otimes \phi =\psi_0 \phi_0 |00\rangle+ \psi_0 \phi_1 |01\rangle +\psi_1 \phi_0 |10\rangle +\psi_1 \phi_1 |11\rangle,
\end{equation}
where we use the standard notation for the products of the basis vectors:
\begin{equation}
\label{eqcomputbastwoqubit}
|00\rangle=|0\rangle \otimes |0\rangle,  \,\,
|01\rangle=|0\rangle \otimes |1\rangle,  \,\,
|10\rangle=|1\rangle \otimes |0\rangle,  \,\,
|11\rangle=|1\rangle \otimes |1\rangle.
\end{equation}
An arbitrary state in $\C^2\otimes \C^2$ can be written as
\begin{equation}
\label{eqgebtwoqubitst}
\xi =\xi_{00} |00\rangle+ \xi_{01} |01\rangle +\xi_{10} |10\rangle +\xi_{11} |11\rangle.
\end{equation}
Comparing with \eqref{eqproducttwoqubitsv} it is seen that \eqref{eqgebtwoqubitst} is not entangled
(is a product state) if and only if
\begin{equation}
\label{eqtwoqubitentangcond}
\xi_{00}\xi_{11} = \xi_{10}\xi_{01}.
\end{equation}

Apart from the standard basis \eqref{eqcomputbastwoqubit} a key role in application is played
by the {\it Bell basis}\index{Bell basis} consisting of fully entangled states:
\begin{equation}
\label{eqBellbasis}
\frac{1}{\sqrt 2}(|00\rangle+|11\rangle), \frac{1}{\sqrt 2}(|00\rangle-|11\rangle),
\frac{1}{\sqrt 2}(|01\rangle+|10\rangle), \frac{1}{\sqrt 2}(|01\rangle-|10\rangle).
\end{equation}
Specific role belongs also to the last vector $(|01\rangle-|10\rangle)/\sqrt 2$, referred to
as the {\it singleton}, because it is rotationally invariant. Namely, the group $SU(2)$ acts naturally in
$\C^2\otimes \C^2$ as $u(\phi \otimes \psi)=u(\phi)\otimes u(\psi)$, and the vector
$(|01\rangle-|10\rangle)/\sqrt 2$ is invariant under this action.

\begin{exer} Check this invariance.
\end{exer}

\begin{remark} The above mentioned action of $SU(2)$ on $\C^2\otimes \C^2$ decomposes into the direct sum
of two irreducible representations, one-dimensional one generated by $(|01\rangle-|10\rangle)/\sqrt 2$
and three-dimensional one, generated by other tree vectors of the Bell basis. Therefore these three vectors
are referred to as a triplet.
\end{remark}

Up to the phase shifts (that is, up to multiplications by numbers of unit magnitude),
arbitrary orthonormal basis in a qubit can be given by the vectors
\begin{equation}
\label{eqqubitbas}
e_0^{\be \phi} =\left(
\begin{aligned}
& \quad \cos(\be/2) \\
& \sin(\be/2) e^{i\phi}
\end{aligned}
\right),
\quad
e_1^{\be \phi} =\left(
\begin{aligned}
& -\sin(\be/2)  e^{-i\phi}\\
& \quad \cos(\be/2)
\end{aligned}
\right),
\end{equation}
which can be obtained by acting on the standard basis by the operator $U_{\be \phi}$:
\[
U_{\be \phi}=\left(
\begin{aligned}
& \quad \cos(\be/2) \quad  -\sin(\be/2)  e^{-i\phi} \\
& \sin(\be/2) e^{i\phi} \quad \quad  \cos(\be/2)
\end{aligned}
\right), \quad U_{\be\phi}e_j=e_j^{\be\phi}, \, j=0,1.
\]
These vectors are eigenvectors of the operator
\[
S_{\be\phi}=U_{\be\phi} \si_3 U^*_{\be\phi}
=\left(
\begin{aligned}
& \quad \cos \be \quad  \sin \be  e^{-i\phi} \\
& \sin \be e^{i\phi} \quad -  \cos\be
\end{aligned}
\right) =\sin \be \cos \phi \si_x +\sin \be \sin \phi \si_y +\cos \be  \si_z,
\]
which can be considered as the projection of the matrix-valued
{\it spin-vector}\index{spin vector}
${\bf \si}=(\si_x, \si_y, \si_z)$ on the unit vector
$n=(\sin \be \cos \phi, \sin \be \sin \phi,\cos \be) \in \R^3$,
and therefore referred to as the {\it component of the spin in the direction}
\index{spin vector!component} $n$.

This basis can be used to demonstrate that the Schmidt decomposition
\eqref{eqSchmidtDecomp} is not unique, but the notion of
maximally entangled state is still well defined. In fact,
one sees directly that maximally entangled  vector
$(|00\rangle+|11\rangle)/\sqrt 2$ can be written analogously in other bases $e_j^{\be \phi}$:
\begin{equation}
\label{eqnotuniqueschmi}
\frac{1}{\sqrt 2} (|00\rangle+|11\rangle)
=\frac{1}{\sqrt 2}\left(e_0^{\be \phi}\otimes \overline{e_0^{\be \phi}}+e_1^{\be \phi}\otimes \overline{e_1^{\be \phi}}\right).
\end{equation}

Also another way to express the rotation invariance of the singleton state is to observe that it has the same
form when expressed in any pair  \eqref{eqqubitbas}:

\begin{equation}
\label{eqnotuniqueschmi1}
\frac{1}{\sqrt 2}(|01\rangle-|10\rangle)
=\frac{1}{\sqrt 2}\left(e_0^{\be \phi}\otimes e_1^{\be \phi}-e_1^{\be \phi}\otimes e_0^{\be \phi}\right).
\end{equation}

\begin{exer} Check \eqref{eqnotuniqueschmi} and \eqref{eqnotuniqueschmi1}.
\end{exer}

\section{Measurement in quantum mechanics}
\label{secquantmes}

Measurements in quantum mechanics occur via an interaction of the measured
quantum system described by the Hilbert space $H$ with another system,
an apparatus, so that this interaction changes the state of the initial system.
{\it Physical observables} are given by self-adjoint matrices $A\in \LC(H)$. Such
matrices $A$ are known to have the spectral decomposition $A=\sum_j \la_j P_j$, where
 $P_j$ are orthogonal projections on the eigenspaces of $A$ corresponding to
 the eigenvalues $\la_j$. According to the {\it basic postulate of quantum measurement}
 \index{basic postulate of quantum measurement}, measuring observable $A$ in a state $\rho$
(often referred to as the {\it Stern-Gerlach experiment}\index{ Stern-Gerlach experiment})
can yield each of the eigenvalue $\la_j$ with the probability
   \begin{equation}
\label{eqantiunitdress1}
{\tr} \, (\rho P_j),
\end{equation}
 and, if the value $\la_j$ was obtained, the state
of the system changes to the reduced state
\[
P_j\rho P_j/ {\tr} \, (\rho P_j).
\]
In particular, if the state $\rho$ was pure, $\rho=|\psi\rangle \langle \psi|$, then the
probability to get $\la_j$ as the result of the measurement becomes $(\psi_,P_j\psi)$
and the reduced state also remains pure and is given by the vector $P_j\psi$.
If the interaction with the apparatus was preformed 'without reading the results',
the state $\rho$ is said to be subject to a {\it non-selective measurement}
\index{non-selective measurement} that changes
$\rho$ to the state $\sum_j P_j\rho P_j$.

\begin{remark} The notion of a general state (a density matrix) arises naturally from the simple
duality \eqref{dualopernorm} (with a bit more nontrivial extension to the
case of infinite-dimensional spaces). In fact, von Neumann introduced a state as a linear
functional on the space of observables (Hermitian linear operators) that was supposed to
show the results of possible measurements of all these observables. By
\eqref{dualopernorm}, this led inevitably to the notion of a mixed state as given above.
\end{remark}

Extended to all self-adjoint operators the transformation
\begin{equation}
\label{eqantiunitdresscondexp}
E:B \mapsto \sum_j P_jB P_j
\end{equation}
is sometimes called the {\it conditional expectation}\index{conditional expectation in quantum systems}
 from $\LC(H)$ to the subalgebra
$N\subset \LC(H)$ of operators that commute with all $P_j$. The conditional expectation
is seen to satisfy the following properties reminiscent to its classical counterpart:
(i) $E(X^*)=[E(X)]^*$, (ii) $X\ge 0$ implies $E(X) \ge 0$, (iii) $E(X)=X$ if and only if $X\in N$,
(iv) If $X_1,X_2 \in N$, then $E(X_1BX_2)=X_1E(B)X_2$ (take-out-what-we-know property).
Another accepted term for \eqref{eqantiunitdresscondexp} is the {\it pinching map}.

For instance, applying the above scheme to the Pauli operator $\si_3$ of a qubit,
allows one to conclude that, assuming the state of a qubit is some $\rho$, the measurement
would reveal the values $1$ or $-1$ corresponding to the pure states $e_0$ or $e_1$, with the
probabilities
\begin{equation}
\label{eqmesqubitsi3}
p_0={\tr}\, (\rho |e_0\rangle \langle e_0|)=\langle e_0|\rho|e_0\rangle =\rho_{00},
\quad p_1=\rho_{11},
\end{equation}
showing in particular that the condition for density matrices to have the unit trace is necessary
to make the probabilistic postulate of quantum measurement consistent. This measurement
also corroborates the interpretation of the states $\rho$ being the mixture of the pure states
$|e_0\rangle \langle e_0|$ and $|e_1\rangle \langle e_1|$ with probabilities $p_1$ and $p_2$:
\begin{equation}
\label{eqmixpuredirect}
\rho= p_0 |e_0\rangle \langle e_0| +p_1 |e_1\rangle \langle e_1|.
\end{equation}
However, all $\rho$ with the same diagonal elements as in \eqref{eqmixpuredirect}
 will give the same result under this measurement.

In representation \eqref{eqBlochspherestate} the probabilities $p_0, p_1$
from \eqref{eqmesqubitsi3} take the values
\begin{equation}
\label{eqmesqubitsi3stok}
p_0=(1+x_3)/2, \quad p_1=(1-x_3)/2.
\end{equation}

More generally, if the state of $n$ dimensional system is $\rho$ and we performed the measurement
of any observable $A$ that is diagonal in the standard basis $e_1, \cdots ,e_n$ of $\C^n$, the probability
to obtain $e_j$ as the result of the measurement will be
\begin{equation}
\label{eqmesqunit}
\langle e_j| \rho | e_j\rangle=\rho_{jj}.
\end{equation}

As seen from this formula all $A$ which are diagonal in the basis $e_1, \cdots , e_n$
and have different eigenvalues produce the same probabilities of finding $e_j$,
Thus effectively we are measuring  the operator that labels the elements of the basis.
This calculation is often referred to as measurement in the {\it computational basis}
$e_1, \cdots , e_n$.
If we are working in the product of two qubits $\C^2\otimes \C^2$, which is the most basic
scene for two-player two-action quantum games, the simplest computational basis is \eqref{eqcomputbastwoqubit}.
For a state $\rho \in \TC(\C^2)\otimes \TC(C^2)$ the measurement in this computational basis
will produce any of this vectors with the probabilities
\begin{equation}
\label{eqprobmesprod}
\langle jk | \rho | jk\rangle ={\tr} (\rho  | jk\rangle \langle jk|),
\end{equation}

For a tensor product of two spaces $H_A$ and $H_B$, with bases $\{e_j^A\}$ and $\{e_j^B\}$,
the matrix elements of the linear operators
(for instance, density matrices) are defined as
\[
\rho (e_i^A\otimes e_j^B)= \sum_{i',j'} \rho_{i'j',ij}(e_{i'}^A\otimes e_{j'}^B),
\]
and the probability to obtain $e_i^A \otimes e_j^B$ with the measurement performed on $\rho$
(in the computational basis $e_i^A \otimes e_j^B$) is the diagonal element $ \rho_{ij,ij}$.
In particular,
\[
(\rho^A \otimes \rho^B)_{i'j', ij}= \rho^A_{i'i} \rho^B_{j'j},
\]
and the probability to obtain $e_i^A \otimes e_j^B$ with the measurement performed on
$\rho^A \otimes \rho^B$ is the product
\begin{equation}
\label{eqcomputbastwoqunit}
\rho^A_{ii} \rho^B_{jj}.
\end{equation}
If we are in the pure state
\[
\rho=|\psi\rangle \langle \psi|,
\quad \psi=\sum_{i,j} \psi_{ij} e_i^A \otimes e_j^B,
\]
then this probability reduces to $|\psi_{ij}|^2$.

As we shall see, most of the quantum games can be ultimately reformulated in classical terms.
However, specific feature lies in the physical realizability of the strategies involved.
For instance, all unitary operators on a single qubit can be realized by the
{\it Mach-Zender interferometer}\index{Mach-Zender interferometer},
which manipulates with photons, whose two states are usually denoted
$|R\rangle =|0\rangle$ and $|L\rangle=|1\rangle$ (for right and left polarization),
and which is built from the following three units (referred to as {\it passive optical devices}
\index{passive optical devices} that form the {\it standard linear optics toolbox}):
(1) {\it beam splitter (BS)}\index{beam splitter} preforming the unitary transformation $U_{BS}$:
\[
U_{BS}= \frac{1}{\sqrt 2}\left(|R\rangle \langle R| +|L\rangle \langle L| +i |R\rangle \langle L|+i|L\rangle \langle R|\right)
=\frac{1}{\sqrt 2}(\1 +i\si_1),
\]
\begin{equation}
\label{eqbeamsplitter}
U_{BS}(|R\rangle)=\frac{1}{\sqrt 2}(|R\rangle+i|L\rangle),
\quad U_{BS} (|L\rangle)=\frac{1}{\sqrt 2}(|L\rangle+i|R\rangle),
\end{equation}
(2) {\it mirror operator}\index{mirror operator} $U_{mir}=-i\si_1$ that takes
$|R\rangle$ to $i|L\rangle$ and  $|L\rangle$ to $i|R\rangle$,
and (3) {\it phase shifters}\index{phase shifter}
\begin{equation}
\label{eqphaseshifters}
U_R(\phi)=  |R\rangle e^{i\phi} \langle R| +|L\rangle \langle L|=\left(\begin{aligned}
& e^{i\phi} \quad 0 \\
& \,\, 0  \quad  0
\end{aligned}
\right),
\quad
U_L(\phi)=  |R\rangle \langle R| +|L\rangle e^{i\phi} \langle L|=\left(\begin{aligned}
& 0 \quad \,\, 0 \\
& 0  \quad e^{i\phi}
\end{aligned}
\right)
\end{equation}

\begin{remark} Sometimes also other optical devices realizing unitary operators are referred
to as beam splitters, for instance
the ideal BS is the rotation $\cos \theta \si_1 +\sin \theta \si_3$ with $\cos \theta$ and $\sin \theta$ referred to
as transmittance and reflectivity parameters respectively.
\end{remark}

The standard combination of these units in the  {\it Mach-Zender setting} acts as
\begin{equation}
\label{eqMachZender}
U_{MZ}=U_R(\theta_2) U_{BS} U_R(\phi_1) U_L(\phi_2) U_{mir} U_{BS} U_R(\theta_1),
\end{equation}
which is easily seen to yield the full $4$-parameter representations of the group $U(2)$
(compare with representation \eqref{eqgenunit2dim2} of $SU(2)$).

Another scheme of physically realizable units that reproduces directly the Cartan
decomposition \eqref{eqgenunit2dim4} can be built from the so-called {\it quarter wave plates}
\index{quarter wave plates} and {\it half wave plates}\index{half wave plates} acting as
\begin{equation}
\label{eqquaterhalfplate}
U_{qwp}(\theta)=e^{-i\theta \si_2}e^{-i\pi \si_3/4}e^{-i\theta \si_2},
\quad
 U_{hwp}(\theta)= [U_{qwp}(\theta)]^2 =e^{-i\theta \si_2}e^{-i\pi \si_3/2}e^{-i\theta \si_2}.
\end{equation}
Combining these two schemes one can naturally build the {\it universal unitary gate for
two-qubit states}, that is, the schemes realizing all possible unitary transformations
of two-qubit states (see detail in \cite{EnglertTwoQubit01}).

Yet another optical devise is the so-called filter, which can be oriented in different
ways to make computations in the bases \eqref{eqqubitbas}. It corresponds to the observable
that is diagonal in this basis and gives values $0$ and $1$ for $e_0^{\be \phi}$ and
$e_1^{\be \phi}$ respectively. Physically it detects a photon if finds it in $e_0^{\be \phi}$
and absorbs it if finds it in the state $e_1^{\be \phi}$. This devise is called a {\it filter}
oriented along the vector  $n=(\sin \be \cos \phi, \sin \be \sin \phi,\cos \be) \in \R^3$.

A key property of the entangled states is that the entanglement is
destroyed when a measurement is carried out on one of the two systems only. For instance,
if, in the Schmidt state \eqref{eqSchmidtDecomp}, we measure the observable $A$
of the first system that is diagonal in the basis $\{\xi_j\}$, that is
$A=\sum_j \la_j P_j$, where $P_j$ is the projection
(in the product space) to the subspace generated by $\xi_j \otimes H_2$,
then we can obtained either of the (not entangled) vectors $\xi_j\otimes \eta_j$ with the
equal probability $1/\min(n,m)$.

This leads one to the far reaching consequence
of non locality of quantum interaction. Namely, suppose that two photons are emitted by some device
in the maximally entangled state \eqref{eqqubitbas} and then move in two different directions.
Suppose we measure one of them by a filter oriented along $n=(\sin \be, 0,\cos \be)$.
According to \eqref{eqqubitbas}, the result will be the states
$e_0^{\be 0}\otimes e_0^{\be 0}/ \sqrt 2$ or $e_1^{\be 0}\otimes e_1^{\be 0}/ \sqrt 2$
with probabilities $\cos^2{\be/2}$ and  $\sin^2{\be/2}$. The remarkable thing is that
this measurement on the first particle affects the other particle, as it brings it to
a well defined state (for any $\be$) and, what is more important,
is independent of the distance (non locality!) between the particles at the time of the measurement.
Thus if after the measurement of the first particle we measure the second particle with the same
filter (oriented along the same vector), we obtain the same result as for the first particle
with probability one (correlation $100\%$). This situation is essentially the famous
{\it Einstein-Podolskii-Rosen (EPR) paradox}\index{Einstein-Podolskii-Rosen (EPR) paradox}
of quantum theory (in its simplified version suggested by D. Bohm). Though Einstein
considered such actions on arbitrary distances as something unnatural,
the recent experiments fully confirmed the conclusions of the EPR thought experiment paving
the path to the experimental work on quantum communication, quantum computation and quantum games.

Measurements arising from self-adjoint operators as described above do not exhaust
the effects of possible measuring instruments. Therefore a more general formulation
of the measurement postulate is needed.
Let $\Om$ be a finite or countable set. A collection of positive operators $\{M_{\om}\}$
in a Hilbert space $H$ parametrized by $\om\in \Om$ (or a mapping from $\Om$ to $\LC^+(H)$) is called
a {\it positive operator valued measure (POVM)}\index{positive operator valued measure (POVM)} if
it satisfies the normalization condition
\begin{equation}
\label{eqdefPOVM}
\sum_{\om\in \Om} M_{\om}=\1.
\end{equation}
The space $\Om$ describes the set of possible outcomes of an experiment.
The {\it measurement based on a POVM} performed in a state $\rho$  produces an outcome $\om$
with the probability
\begin{equation}
\label{eqdefPOVM1}
{\tr} (\rho M_{\om}).
\end{equation}
If all $M_{\om}$ are orthogonal projections, the POVM is called the
{\it projection valued measure (PVM)}\index{projection valued measure (PVM)}.
Only the PVMs arise from self-adjoint operators as described above.

\section{Meyer's quantum penny-flip game}
\label{secMeyGame}

 Let us now introduce the first quantum game proposed by D. Meyer in \cite{MeyerD99}.
 It is an example of a quantum {\it sequential games}, where players act in some order
 on one and the same devise, and represents a version of the penny flip-over game.
The classical setup is as follows. The referee places a coin (penny) head up in a
box. Then three moves are performed sequentially by the two players, $P$ (assumed
to play by the rule of classical probability) and $Q$ (which next will be assumed
 to play by the quantum rules). First $Q$ makes a move by either flipping a coin
 (action $F$) or not (action $N$). Then $P$, not seeing the result of the action
  of the $Q$, makes her move by either flipping it (action $F$) or not (action $N$),
  and finally $Q$ (not seeing what $P$ has done) has the right to flip it again
(action $F$) or not (action $N$). Then the referee opens the box. If the coin is
head up, $Q$ wins and $P$ pays $Q$ a penny. Otherwise $P$ wins and $Q$ pays $P$
a penny. This is a zero-sum game with the table  (the numbers in the table are the payoffs of $P$)
\[
\begin{tabular}{cc}
&Q\\P&
\begin{tabular}{|c|c|c|c|c|}
\hline
&NN&NF&FN& FF\\
\hline
N &-1&1&1& -1\\
\hline
F&1&-1&-1&1\\
\hline
\end{tabular}
\\ &
\\ & {\bf Table 1.1}
\end{tabular}
\]

Easy to see that under the usual rules the optimal minimax value of the game is
$0$ and the optimal minimax strategies of the players are to choose their
strategies uniformly (with probability $1/2$ for $P$ and probability $1/4$ for $Q$).

To construct a quantum version of the game, one augments the two-state classical system
to the qubit by associating the basis vectors
$e_0=|0\rangle=|H\rangle$ and  $e_1=|1\rangle=|T\rangle$ of the qubit $\C^2$
with the states $H$ and $P$. Thus
pure quantum states are unit vectors $|\psi\rangle =a|H\rangle +b|T\rangle$,
$|a|^2+|b|^2=1$ (more precisely, the corresponding elements of the projective
 space $\C\P^1$) and the mixed quantum states are given by the density matrices
\eqref{eqBlochspherestate}.

Possible (pure) transformations of quantum states are given by the unitary matrices
\eqref{eqgenunit2dim1}. However, Meyer has chosen to work with unitary matrices
   \begin{equation}
\label{eqMeyerUnitary}
U=U(u,v)= \left(\begin{aligned}
& u  \quad \,\,\,\, \bar v \\
& v  \quad -\bar u
\end{aligned}
\right), \quad |u|^2+|v|^2=1,
\end{equation}
that, unlike $SU(2)$, have the determinant $-1$.

The opening of a state means the act of its measurement. The
result of such action on a state $\rho$ will be $|H\rangle$ or $T\rangle$
with probabilities \eqref{eqmesqubitsi3stok}.

Flip and no-flip actions $F$ and $N$ are thus presented by the Pauli matrix
$F=\si_1$ and the unit matrix $N=I$:

Let us now assume that $P$ can play classical strategies, $F$ and $N$,
and their classical mixtures, i.e. apply $F$ with some probability
$p\in [0,1]$ and $N$ with the probability $1-p$. The key point (or assumption) in the
quantum setting is that classical mixtures randomize the actions on the
density matrices, that is, the mixed $p$-strategy of $P$ acts on a density
matrix $\rho$ by the rule
\[
\rho \mapsto pF\rho F^*+(1-p) N\rho N^*.
\]

The quantum player $Q$ is supposed to play by (pure) {\it quantum strategies},
that is, by applying arbitrary unitary operators $U(u,v)$. Thus, after the
first move of $Q$ the initial state
\[
\rho_0=|0\rangle \langle 0|=|H\rangle \langle H|=\left(\begin{aligned}
& 1  \quad 0 \\
& 0  \quad 0
\end{aligned}
\right)
\]
turns to the state
\[
\rho_1=U(u,v) \rho_0 U^*(u,v) = \left(\begin{aligned}
& u \bar u \quad u\bar v \\
& v \bar u \quad v\bar v
\end{aligned}
\right).
\]

After the move of $P$ the state turns to
\[
\rho_2= pF\rho_1 F^*+(1-p) N\rho_1 N^*
= \left(\begin{aligned}
& pv \bar v+(1-p)u\bar u \quad pv\bar u+(1-p)u\bar v \\
& pu \bar v +(1-p)v\bar u \quad pu\bar u +(1-p)v\bar v
\end{aligned}
\right).
\]

If the game would stop here, the payoff to the player $P$ would be
\begin{equation}
\label{eqMeypay2}
(-1)[ pv \bar v+(1-p)u\bar u] +  pu\bar u +(1-p)v\bar v
=(2p-1)(u\bar u-v\bar v).
\end{equation}

It is seen that the game with such payoff has the value:
\[
\max_p \min_{u,v}[(2p-1)(u\bar u-v\bar v)]=\min_{u,v} \max_p[(2p-1)(u\bar u-v\bar v)]=0,
\]
and the minimax strategies of the players are $p=1/2$, $u^2=v^2=1/2$.

To see what happens if the third move is included, assume that $Q$ plays both times
with the {\it Hadamard matrix} or {\it Hadamard gate} $U=U(1/\sqrt 2, 1/\sqrt 2)$. Then
\[
\rho_1=  \left(\begin{aligned}
& 1/2 \quad 1/2 \\
& 1/2 \quad 1/2
\end{aligned}
\right)
=|\psi\rangle \langle \psi|
\]
with
\[
\psi =(H\rangle +T\rangle )/\sqrt 2.
\]
This state can be thought of imaginatively as describing the coin standing on its side. Then
$\rho_2=\rho_1$ independently of the choice of $p$, and thus
\[
\rho_3=  U(1/\sqrt 2, 1/\sqrt 2)  \left(\begin{aligned}
& 1/2 \quad 1/2 \\
& 1/2 \quad 1/2
\end{aligned}
\right)  U^*(1/\sqrt 2, 1/\sqrt 2)
\]
\[
= |U(1/\sqrt 2, 1/\sqrt 2) \psi \rangle \langle  U(1/\sqrt 2, 1/\sqrt 2)\psi|=
|H\rangle \langle H|= \left(\begin{aligned}
& 1 \quad 0 \\
& 0 \quad 0
\end{aligned}
\right),
\]
so that $Q$ wins with probability $1$ independently of the actions of $P$! The power
of quantum strategies (or quantum communications) is thus explicitly revealed.

In paper  \cite{MeyerD99} one can also find the discussion of what can happen
if both players are allowed to use pure quantum or even mixed quantum strategies.

\section{First sequential games: quantum Prisoner's dilemma}

In this and the next sections we present the two basic approaches proposed
for the quantization of the simultaneous static games, firstly the so-called {\it EWL protocol}
suggested in  \cite{EWL99} on the example of Prisoner's dilemma and secondly the so-called
{\it MW protocol} suggested in \cite{MW00} on the example of the Battle of the Sexes.

The table of the Prisoner's dilemma worked with in \cite{EWL99} was as follows:
\[
\begin{tabular}{cc}
&Bob\\Alice&
\begin{tabular}{|c|c|c|}
\hline
&C&D\\
\hline
C & (3,3) & (0,5)\\
\hline
D & (5,0) & (1,1) \\
\hline
\end{tabular}
\\ &
\\ & {\bf Table 1.2}
\end{tabular}
\]
which is a performance of a more general version
\[
\begin{tabular}{cc}
&Bob\\Alice&
\begin{tabular}{|c|c|c|}
\hline
&C&D\\
\hline
C & (r,r) & (s,t)\\
\hline
D & (t,s) & (p,p) \\
\hline
\end{tabular}
\\ &
\\ & {\bf Table 1.3}
\end{tabular}
\]
with $r$ for reward, $p$ for punishment, $s$ for sucker's payoff, where $t>r>p>s$.

In the quantum version each player can manipulate a qubit (rather than playing with
two bits in the classical version) generated by the basis of two vectors that now are identified
with the actions of cooperate or defect: $e_0=|0\rangle =|C\rangle$, $e_1=|1\rangle =|D\rangle$.
Only the referee has access to the combine system of two qubits and she prepares the initial state
$\psi_{in}=J |CC\rangle$ with some unitary operator $J$ in $\C^2\otimes \C^2$ (made known to both
 players), which is symmetric with respect to the interchange of the players. Physically $J$ is said to act as
 the {\it entanglement} that mixes in some way the initial product form $|CC\rangle=|C\rangle \otimes |C\rangle$.
Then Alice and Bob choose (simultaneously and independently) some unitary operators $U_A$ and $U_B$
in $\C^2$ to act on their qubits, which transform the state $\psi_{in}$ into $(U_A\otimes U_B) \psi_{in}$.
Finally the referee redoes the entanglement by applying $J^*=J^{-1}$ yielding the final state
  \begin{equation}
\label{eqEWLprot}
|\psi_{fin} \rangle =J^* (U_A\otimes U_B) J|CC\rangle ,
\end{equation}
often referred to as {\it EWL protocol}. For
\[
|\psi_{fin}\rangle
= \psi_{CC} |CC\rangle + \psi_{CD} |CD\rangle +\psi_{DC} |DC\rangle +\psi_{DD} |DD\rangle,
\]
the squares  $|\psi_{CC}|^2$, $|\psi_{CD}|^2$, $|\psi_{DC}|^2$, $|\psi_{DD}|^2$ are the
probabilities of the corresponding outcomes, so that the payoffs for Alice and Bob are
  \begin{equation}
\label{eqEWLpayAlBob}
\Pi_A=r|\psi_{CC}|^2 +p |\psi_{DD}|^2+ t |\psi_{DC}|^2+ s|\psi_{CD}|^2,
\Pi_B=r|\psi_{CC}|^2 +p |\psi_{DD}|^2+ s |\psi_{DC}|^2+ t|\psi_{CD}|^2.
\end{equation}

Clearly the game depends on both the choice of the entangling operator $J$ and the set of
allowed unitary operators, that is, the strategy spaces of Alice and Bob.

Concretely, in \cite{EWL99}, the set of unitary operators used by both players was restricted,
rather artificially in fact, to the two-parameter set
\begin{equation}
\label{eqEWLstrat}
U(\theta, \phi) =
  \left(\begin{aligned}
& e^{i\phi} \cos (\theta/2) \quad \quad \sin (\theta /2) \\
& -\sin (\theta/2)  \quad  e^{-i\phi} \cos (\theta/2)
\end{aligned}
\right)
\end{equation}
with $0\le \theta \le \pi$, $0\le \phi \le \pi/2$.
In particular, the operators $\hat C$ and  $\hat D$,
\[
\hat C =U(0,0)=  \left(\begin{aligned}
& 1 \quad 0 \\
& 0 \quad 1
\end{aligned}
\right)=\1, \quad
\hat D =U(\pi,0)=  \left(\begin{aligned}
& \quad 0 \quad 1 \\
& -1 \quad 0
\end{aligned}
\right),
\]
 were associated with the cooperative and defective classical strategies, because as the starting point
 was supposed to be $|CC\rangle$, the identity operator $\hat C$ preserves the cooperative behavior
 $|C\rangle$ and $\hat D$ flips it to the defective behavior $|D\rangle$.

\begin{remark}
In later publications the changing sign feature of $\hat D$ (looking a bit artificially)
was mostly abandoned and one used the exact flipping operator $F=\si_x$ instead.
\end{remark}

The assumptions on $J$ made in \cite{EWL99} were introduced with a very clear interpretation,
as those that would allow to reproduce the classical game. Namely, the commutativity conditions
  \begin{equation}
\label{eqEWLprotcond}
[J, \hat D\otimes \hat D]=0, \quad  [J, \hat D\otimes \hat C]=0, \quad [J, \hat C\otimes \hat D]=0,
\end{equation}
were assumed, implying that
 \begin{equation}
\label{eqEWLprotcond1}
[J, U(\theta,0)]=0
\end{equation}
for all $\theta$. If this holds, then
\[
\psi_{fin}\rangle= U_A \otimes U_B |CC\rangle
=[\cos(\theta_A/2) |C\rangle-\sin(\theta_A/2) |D\rangle]
\otimes
[\cos(\theta_B/2) |C\rangle-\sin(\theta_B/2) |D\rangle],
\]
and all probabilities factorize,
whenever $U_A$ and $U_B$ are restricted to $U(\theta,0)$, in particular, if
$U_A$ and $U_B$ are allowed to be only the 'classical actions', i.e. either $\hat C$
or $\hat D$. Thus, identifying $\cos^2(\theta_A/2)$
and $\cos^2(\theta_B/2)$ with classical probabilities $p$ and $q$ we reproduce the
payoffs of the classical prisoners' dilemma played with the mixed strategies.

\begin{exer}
(i) Check that in the basis $|CC\rangle, |CD\rangle, |DC\rangle, |DD\rangle$
\[
\hat C \times \hat D=-\left(\begin{aligned}
& \hat D \quad 0 \\
& 0 \quad \hat D
\end{aligned}
\right),
\quad
\hat D \times \hat C=\left(\begin{aligned}
& 0 \quad -1 \\
& 1 \quad \quad 0
\end{aligned}
\right),
\quad
\hat D \times \hat D=\left(\begin{aligned}
& \quad 0 \quad \hat D \\
& -\hat D \quad  0
\end{aligned}
\right),
\quad
\si_x \times \si_x=\left(\begin{aligned}
& \,\, 0 \quad \si_x \\
& \si_x  \quad \,\, 0
\end{aligned}
\right).
\]

(ii) For $J=\left(\begin{aligned}
& A \quad B \\
& R \quad S
\end{aligned}
\right)$, condition \eqref{eqEWLprotcond} is equivalent to
\[
 J=\left(\begin{aligned}
& \quad A \quad B \\
& -B \quad A
\end{aligned}
\right), \quad A\hat D=\hat D A, \quad  B\hat D=\hat D B
\]
(and under $[J, C\otimes D]=0$ the first two conditions of \eqref{eqEWLprotcond}
are equivalent). Hence $J$ is a linear combination of the matrices
\[
\1, \hat C \times \hat D, \hat D \times \hat C, \hat D \times \hat D.
\]
\end{exer}

In \cite{EWL99} the operator $J$ was chosen as
\[
J_{\ga}=\exp\{ i\ga \hat D \otimes \hat D/2\}, \quad \ga \in [0, \pi/2],
\]
so that
\[
J_{\ga}=\cos(\ga/2) \hat C\otimes \hat C +i\sin (\ga/2) \hat D\otimes \hat D.
\]
(The choice of coefficients is also restricted by the requirement that $J$ is unitary.)

\begin{remark} This choice of $J$ can be considered as the most general fully
symmetric choice. In future publications, when the flipping $F=\si_x$ became
standard substitute to $\hat D$, the version with
\[
J_{\ga}=\cos(\ga/2) \1\otimes \1 +i\sin (\ga/2) \si_x\otimes \si_x
\]
became a more or less canonical choice.
\end{remark}

We have arrived now at the problem of finding Nash equilibria for a
 purely classical stationary game of two players
with payoffs \eqref{eqEWLpayAlBob}, where $\psi_{fin}$ is calculated from
 \eqref{eqEWLprot} based on the players strategies, which are the pairs
 $\theta_A, \phi_A$ of Alice and $\theta_B, \phi_B$ of Bob, defining
 $U_A$ and $U_B$ via \eqref{eqEWLstrat}. The whole quantum content is encoded
 in the particular way the payoffs are calculated.

If $\ga=0$, $J_0$ is the identity operator and the probabilities factorize for
all $U_A, U_B$. Hence $\hat D\otimes \hat D$ is equilibrium in dominated strategies,
as in classical game. Therefore $\ga $ is considered as the 'entangling parameter',
and the case with $\ga= \pi/2$ as the 'maximally entangled game'. In this case
\[
J_{\pi/2}=\frac{1}{\sqrt 2}[ \hat C\otimes \hat C +i \hat D\otimes \hat D],
\quad J_{\pi/2} |CC\rangle =\frac{1}{\sqrt 2}(|CC\rangle+i|TT\rangle),
\]
and $\hat D\otimes \hat D$ is not a Nash equilibrium.
The calculations show (see \cite{EWL99} and a very detailed presentation in
review \cite{GuoZhang08}) that in the maximally entangled game there is a unique Nash equilibrium
$\hat Q\otimes \hat Q$ with
\[
\hat Q=i\si_3= \left(\begin{aligned}
&  i \quad \quad 0 \\
& 0 \quad -i
\end{aligned}
\right)
\]
with the payoff $(3,3)$, that is, the Pareto optimal (effective) solution became
the unique Nash equilibrium!

As was noted in  \cite{BH01}, this effect was achieved precisely by the artificial
restriction of the strategy space to operators \eqref{eqEWLstrat}.
The same effect can be achieved just with three strategies $I, \si_y=\si_2, \si_z=\si_3$ for each player
with $\si_z$ giving the desired equilibrium. In \cite{GuoZhang08} detailed calculations
are given reproducing the same effect from the three strategies $\1, \si_x, \si_z$.
But if we allow the full discrete set $\1, \si_x, \si_y, \si_z$ the equilibrium disappears.
This is in fact a consequence of a general result, see below Proposition \ref{noNash quantumtwo}.

\section{First sequential games: quantum Battle of the Sexes}

A slightly different approach to the 'quantization of games' was proposed in \cite{MW00}
on the example of the Battle of the Sexes, namely the game given by the table
\[
\begin{tabular}{cc}
&Bob\\Alice&
\begin{tabular}{|c|c|c|}
\hline
&O&T\\
\hline
O & ($\al, \be$) & ($\ga, \ga$)\\
\hline
T & ($\ga, \ga$) & ($\be, \al$) \\
\hline
\end{tabular}
\\ &
\\ & {\bf Table 1.4}
\end{tabular}
\]
where $\al>\be>\ga$. Here $O$ reflects the preferred activity of the wife (opera, ballet, etc)
 and $T$ that of the husband (television, football, etc). Again each player has a quibit $\C^2$ at their
 disposal with the two basic states now denoted $e_0=|O\rangle$, $e_1=|T\rangle$.

Classical theory (see e.g. \cite{KolMal})
yields the conclusion that this game has two pure Nash equilibria  $(O,O)$, $(T,T)$
and one mixed equilibrium ($p$ and $q$ denote the probabilities used by Alice and Bob to play $O$):
\begin{equation}
\label{eqbattleSex}
p^*=\frac{\al-\ga}{\al+\be-2\ga}, \quad q^*=1-p^*=\frac{\be-\ga}{\al+\be-2\ga},
\end{equation}
the payoff for both Alice and Bob in this equilibrium being
\begin{equation}
\label{eqbattleSex1}
(\al+\be)p^*q^*+\ga(p^*p^*+q^*q^*)=\frac{\al \be -\ga^2}{\al+\be-2\ga}<\min(\al, \be).
\end{equation}

As the simplest possible quantum version one can suggest that instead of choosing probabilities
$p,q$ to play the classical strategies, Alice and Bob are allowed to choose quantum superpositions,
that is, the pure quantum states $a |O\rangle +b |T\rangle$ and $c |O\rangle +d |T\rangle$
respectively,  with $|a|^2+|b|^2=1$ and $|c|^2+|d|^2=1$ and the outcome is then measured
according to the measurement rules of quantum mechanics (see the end of Section \ref{secpreliomfinitequ}).
Namely, their common pure state in $\C^2\otimes \C^2$ becomes
\[
\psi_{in}= (a |O\rangle +b |T\rangle)\otimes (c |O\rangle +d |T\rangle)
=ac |OO\rangle +ad |OT\rangle +bc |TO\rangle + bd |TT\rangle,
\]
so that after the measurements one gets $|OO\rangle$ with probability $|ac|^2$,
$|OT\rangle$ with probability $|ad|^2$, $|TO\rangle$ with probability $|bc|^2$
and  $|TT\rangle$ with probability $|bd|^2$. This is exactly the same result,
as if they play classical mixed strategies choosing $O$ with probabilities
$p=|a|^2$ and $q=|c|^2$ respectively, so that this quantum version of the game
reproduces the classical game.

Introducing more advanced quantum operations one can assume that
the players start at some initial pure state $\psi_{in}$ (the analog
of $J|CC\rangle$ of the EWL protocol), or even
mixed state $\rho_{in}$, and then Alice and Bob are allowed to perform
on their parts of the product $\C^2\otimes C^2$ some quantum operations.
In \cite{MW00} it was suggested that the allowed strategies for Alice
and Bob are either the identity operator $I=\1$ or the exchange (flip)
operator $F=\si_1$, or their classical mixtures, that is, choosing $I$
with probabilities $p$ and $q$ respectively (and thus $F$ with probabilities
$1-p$ and $1-q$). The point to stress is that, like in Meyer's penny flipping
 game, these probabilities are applied to the action of $I$ and $F$ on the
 density matrices by dressing \eqref{eqdressingdef}.

As one can expect,  the situation will be quite different
depending on the initial $\psi_{in}$ or $\rho_{in}$.

Suppose first that they start with a factorizable initial density matrix
$\rho_{in}=\rho_A \otimes \rho_B$. Applying their eligible  mixtures given
by the probabilities $p$ and $q$, Alice and Bob transform the initial state
into the final $\rho_A^{fin} \otimes \rho_B^{fin}$ with
\[
\rho_A^{fin}=pI \rho_A I+(1-p)F\rho_A F,
\quad \rho_B^{fin}=qI \rho_B I+(1-p)F\rho_B F.
\]
For a general initial state the final state becomes
\[
\rho_{fin}= pq (I_A \otimes I_B) \rho_{in} (I_A \otimes I_B)
+p(1-q) (I_A \otimes F_B) \rho_{in} (I_A \otimes F_B)
\]
\begin{equation}
\label{eqMWprot}
+(1-p)q (F_A \otimes I_B) \rho_{in} (F_A \otimes I_B)
+(1-p)(1-q) (F_A \otimes F_B) \rho_{in} (F_A \otimes F_B).
\end{equation}

If they start with the initial density matrix $\rho_{in}$ arising from
the pure states $|OO\rangle$, $|OT\rangle$, $|TO\rangle$, $|TT\rangle$,
then we get back the same classical outcome as in the first simplest scenario.
In fact, starting for instance with $|OO\rangle$,
one obtains
\[
\rho_A^{fin}\otimes \rho^{fin}_B=\left[p \left(\begin{aligned}
& 1 \quad 0 \\
& 0 \quad 0
\end{aligned}
\right)+(1-p) \left(\begin{aligned}
& 0 \quad 0 \\
& 0 \quad 1
\end{aligned}
\right)\right]\otimes
\left[q \left(\begin{aligned}
& 1 \quad 0 \\
& 0 \quad 0
\end{aligned}
\right)+(1-q) \left(\begin{aligned}
& 0 \quad 0 \\
& 0 \quad 1
\end{aligned}
\right)\right]
\]
\[
=\left(p|O\rangle \langle O|+(1-p) |T\rangle \langle T|\right)
\otimes \left(q|O\rangle \langle O|+(1-q) |T\rangle \langle T|\right)
\]
\[
=pq |OO\rangle \langle OO|+p(1-q) |OT\rangle \langle OT|
+(1-p)q |TO\rangle \langle TO|+(1-p)(1-q) |TT\rangle \langle TT|.
\]
This density matrix yields outcomes $|OO\rangle$,  $|OT\rangle$,  $|TO\rangle$  $|TT\rangle$
with the probabilities $pq$, $p(1-q)$, $(1-p)q$ and $(1-p)(1-q)$, respectively.

The situation changes if we start  with $\psi_A=  a |O\rangle +b |T\rangle$
and $ \psi_B=c |O\rangle +d |T\rangle$, so that
\[
|\psi_{in}\rangle =\psi_A \otimes \psi_B=
ac |OO\rangle+ad |OT\rangle +bc |TO\rangle +ad |TT\rangle
=c_{OO} |OO\rangle+c_{OT} |OT\rangle +c_{TO} |TO\rangle +c_{TT} |TT\rangle,
\]
or, in terms of the density matrices,
\[
\rho_{in}=\rho_A \otimes \rho_B,
\]
\[
\rho_A =|\psi_A\rangle \langle \psi_A|= \left(\begin{aligned}
& a\bar a \quad a \bar b \\
& b \bar a \quad b \bar b
\end{aligned}
\right),
\quad
\rho_B =|\psi_B\rangle \langle \psi_B|= \left(\begin{aligned}
& c\bar c \quad c \bar d \\
& d \bar c \quad d \bar d
\end{aligned}
\right).
\]

Applying  $I$ with probabilities $p$ and $q$ yields
\[
\rho_A^{fin} =
p \left(\begin{aligned}
& a\bar a \quad a \bar b \\
& b \bar a \quad b \bar b
\end{aligned}
\right)+(1-p) \left(\begin{aligned}
& b \bar b \quad b \bar a \\
& a \bar b \quad a \bar a
\end{aligned}
\right),
\quad
 \rho_B^{fin} =
q \left(\begin{aligned}
& c\bar c \quad c \bar d \\
& d \bar c \quad d \bar d
\end{aligned}
\right)+(1-q) \left(\begin{aligned}
& d \bar d \quad d \bar c \\
& c \bar d \quad c \bar c
\end{aligned}
\right).
\]

According to \eqref{eqcomputbastwoqunit},
the probabilities to get $|OO\rangle$, $|OT\rangle$, etc,  become
 \[
P_{OO}= [pa\bar a+(1-p) b\bar b][qc\bar c +(1-q) d\bar d]
\]
 \begin{equation}
\label{eqMWoutcomepr1}
 =pq |c_{OO}|^2+(1-p)q |c_{TO}|^2+p(1-q) |c_{OT}|^2+(1-p)(1-q) |c_{TT}|^2,
  \end{equation}
 \[
 P_{TT}= [pb\bar b+(1-p) a\bar a][qd\bar d +(1-q) c\bar c]
\]
 \begin{equation}
\label{eqMWoutcomepr2}
=pq |c_{TT}|^2+(1-p)q |c_{OT}|^2+p(1-q) |c_{TO}|^2+(1-p)(1-q) |c_{OO}|^2,
  \end{equation}
 \[
 P_{OT}= [pa\bar a+(1-p) b\bar b][qd\bar d +(1-q) c\bar c]
\]
 \begin{equation}
\label{eqMWoutcomepr3}
 =pq |c_{OT}|^2+(1-p)q |c_{TT}|^2+p(1-q) |c_{OO}|^2+(1-p)(1-q) |c_{TO}|^2,
  \end{equation}
 \[
 P_{TO}= [pb\bar b+(1-p) a\bar a][qc\bar c +(1-q) d\bar d]
 \]
  \begin{equation}
\label{eqMWoutcomepr4}
=pq |c_{TO}|^2+(1-p)q |c_{OO}|^2+p(1-q) |c_{TT}|^2+(1-p)(1-q) |c_{OT}|^2,
 \end{equation}
 which are different from the classical outcomes.

 One can expect to have the same outcomes for the general initial superposed state:
\begin{equation}
\label{eqMWinitial}
|\psi_{in}\rangle= \sum_{i,j=0}^1c_{ij} |ij\rangle
=c_{OO} |OO\rangle+c_{OT} |OT\rangle +c_{TO} |TO\rangle +c_{TT} |TT\rangle.
\end{equation}

Let us check it. The corresponding density matrix is
\[
\rho_{in}=      |\psi_{in}\rangle \langle \psi_{in}|
=\sum_{i,j,k,l,} c_{ij} \bar c_{kl} |ij\rangle \langle kl|
=\sum_{i,j,k,l,} c_{ij} \bar c_{kl}  |i\rangle \langle k| \otimes  |j\rangle \langle l|.
\]

Hence,
\[
I_A \rho_{in} I_A=I_B \rho_{in}I_B=\rho_{in},
\]
\[
F_A |ij\rangle \langle kl|F_A= F_A|i\rangle \langle k|F_A \otimes  |j\rangle \langle l|
=|i'\rangle \langle k'| \otimes  |j\rangle \langle l|=|i'j\rangle \langle k'l|,
\]
\[
F_B |ij\rangle \langle kl|F_B=|ij'\rangle \langle kl'|,
\quad (F_A\otimes F_B) |ij\rangle \langle kl|(F_A\otimes F_B)
=|i'j'\rangle \langle k'l'|,
\]
where prime denotes the complementary index.

Hence, by \eqref{eqMWprot},
\[
\rho_{fin}= \sum_{i,j,k,l,} c_{ij} \bar c_{kl} (pq  |ij\rangle \langle kl|
+p(1-q)  |ij'\rangle \langle kl'|
\]
\begin{equation}
\label{eqMWprotsep}
+(1-p)q  |i'j\rangle \langle k'l| +(1-p)(1-q)  |i'j'\rangle \langle k'l'|).
\end{equation}

Probability to measure $|ij\rangle$ is the diagonal element
\begin{equation}
\label{eqMWprotsep1}
P_{ij}=(\rho_{fin})_{ij,ij}=pq|c_{ij}|^2  +(1-p)q |c_{i'j}|^2 +p(1-q) |c_{ij'}|^2 +(1-p)(1-q)|c_{i'j'}|^2,
\end{equation}
which is a compact form of equations \eqref{eqMWoutcomepr1}-\eqref{eqMWoutcomepr4} above.

If we now consider the general bi-matrix game with the matrix

 \[
\begin{tabular}{cc}
&Bob\\Alice&
\begin{tabular}{|c|c|c|}
\hline
&O&T\\
\hline
O & ($\al_{00}, \be_{00}$) & ($\al_{01}, \be_{01}$)\\
\hline
T & ($\al_{10}, \be_{10}$) & ($\al_{11}, \be_{11}$) \\
\hline
\end{tabular}
\\ &
\\ & {\bf Table 1.5}
\end{tabular}
\]
we get the payoffs for the first player
\[
\Pi^A=\sum_{ij} \al_{ij}P_{ij}=\sum \al_{ij}[pq|c_{ij}|^2
+(1-p)q |c_{ij'}|^2 +p(1-q) |c_{ij'}|^2 +(1-p)(1-q)|c_{i'j'}|^2]
\]
\begin{equation}
\label{eqMWgenpaytr}
=pq \sum \al_{ij}|c_{ij}|^2+(1-p)q \sum \al_{ij}|c_{ij'}|^2
+p(1-q) \sum \al_{ij}|c_{i'j}|^2 +(1-p)(1-q) \sum \al_{ij}|c_{i'j'}|^2,
\end{equation}
and the same (with $\be$ instead of $\al$) for the second player.

Thus, as the result, applying quantum rules in this {\it MW protocol}, means effectively just
applying the parametric family of transformations to the initial payoff matrix:
$\al \mapsto \tilde \al$,  $\be \mapsto \tilde \be$:
\begin{equation}
\label{eqMWgenpaytr1}
\tilde \al=\left(\begin{aligned}
& \sum \al_{ij}|c_{ij}|^2 \quad  \sum \al_{ij}|c_{i'j}|^2 \\
&  \sum \al_{ij}|c_{ij'}|^2 \quad \sum \al_{ij}|c_{i'j'}|^2
\end{aligned}
\right), \quad
\tilde \be=\left(\begin{aligned}
& \sum \be_{ij}|c_{ij}|^2 \quad  \sum \be_{ij}|c_{i'j}|^2 \\
&  \sum \be_{ij}|c_{ij'}|^2 \quad \sum \be_{ij}|c_{i'j'}|^2
\end{aligned}
\right).
\end{equation}

In case of the 'most entangled' initial state $\psi=a|OO\rangle+b|TT\rangle$,
$|a|^2+|b|^2=1$, it simplifies to
 \begin{equation}
\label{eqMWgenpaytr2}
\tilde \al=\left(\begin{aligned}
& \al_{00}|a|^2+\al_{11}|b|^2 \quad  \al_{10}|a|^2+\al_{01}|b|^2 \\
&  \al_{01}|a|^2+\al_{10}|b|^2 \quad \al_{11}|a|^2+\al_{00}|b|^2
\end{aligned}
\right), \quad
\tilde \be=\left(\begin{aligned}
& \be_{00}|a|^2+\be_{11}|b|^2 \quad  \be_{10}|a|^2+\be_{01}|b|^2 \\
&  \be_{01}|a|^2+\be_{10}|b|^2 \quad \be_{11}|a|^2+\be_{00}|b|^2
\end{aligned}
\right).
\end{equation}

In particular, the matrix  of Table 1.4 transforms to the matrix

\[
\begin{tabular}{cc}
&Bob\\Alice&
\begin{tabular}{|c|c|c|}
\hline
&O&T\\
\hline
O & ($\tilde \al, \tilde \be$) & ($\tilde \ga, \tilde \ga$)\\
\hline
T & ($\tilde \ga, \tilde \ga$) & ($\tilde \be, \tilde \al$) \\
\hline
\end{tabular}
\\ &
\\ & {\bf Table 1.5}
\end{tabular}
\]
with
\[
\tilde \al=\al |a|^2 +\be |b|^2, \tilde \be= \al |b|^2 +\be |a|^2,
\quad \tilde \ga=\ga.
\]

Now the value of the mixed Nash equilibrium is still less than the
payoffs at both pure equilibria, $p^*=q^*=0$ and $p^*=q^*=1$. These
pure equilibria give payoffs $(\tilde \al,\tilde \be)$ and
$(\tilde \be, \tilde \al)$ respectively. In the special case of
$a=b=1/\sqrt 2$, the payoffs for Alice and Bob coincide in both pure
equilibria. They equal $(\al+\be)/2$, are efficient (Pareto optimal),
and get better payoff than the third equilibrium payoff $(\al+\be +2\ga)/4$.
Hence it was argued in \cite{MW00} that both equilibria $p^*=q^*=0$
and $p^*=q^*=1$ represent somehow the unique solution and thus solve
the dilemma of the Battle of the Sexes. This is of course arguable.
As was commented in \cite{Ben00}, there remain the possibility of mismatch
 (one chooses $0$ and another $1$) giving  lower payoff, which leaves
essentially the same dilemma as the initial classical one.  Marinatto
and Weber argued back that it was natural  for players to stick to
$p^*=q^*=1$, which means doing nothing, rather than start flipping.

\section{Variations on MW protocol}

Using transformation \eqref{eqMWgenpaytr1} one can automatically transform
any game to a new quantum version obtained by the MW protocol.

This transforms, of course, in a systematic way, all properties of the games:
equilibria, their stability, etc. For instance, stability of the equilibria
of the transformed RD for two-player two-action games was analyzed
in \cite{IqbalToRD}, ESS stability for the transformed Rock-Paper-Scissors
game in \cite{IqbalTo02}, and for 3 player games in \cite{IqbalTo02a}. The
transformations of the simplest cooperative games were analyzed in
\cite{IqbalTo02b}.

Transformation \eqref{eqMWgenpaytr1} extends directly to games with arbitrary number
of players and arbitrary number of strategies (in order to preserve
the dimension under this transform, if we have $n$ strategies for a player,
then one must choose exactly $n$ basic transformations allowing to reshuffle
them (say, $n$ transforms taking the first strategy to any of
the $n$ existing strategies).

One can extend the setting of MW protocol by allowing arbitrary
unitary strategies of the players (rather than just $I$ and $F$) and their
classical mixtures.

If the same extension performed with the EWL protocol,
the only real difference between MW and EWL approaches lies in the
application by EWL protocol the disentangling operator $J^*$ before
the measurement, which is not the case in the MW protocol.

Let us review couple of the extensions
performed along these lines and their conclusions.

In \cite{DuLi01}
the MW protocol is applied to the Battle of the Sexes starting from
the 'maximally entangled' initial state
\[
|\psi_{in}\rangle=J|OO\rangle =(|OO\rangle+|TT\rangle)\sqrt 2,
\quad \rho_{in}=|\psi_{in}\rangle \langle \psi_{in}|.
\]
while the strategy space of the players is taken to consist of arbitrary
 unitary $U_A$ and $U_B$, given by the matrices (see \eqref{eqgenunit2dim4})
   \[
U= \left(\begin{aligned}
& \,\, e^{i(\phi+\psi)} \cos \theta  \quad  \,\, ie^{i\phi-\psi} \sin \theta \\
& ie^{-i(\phi-\psi)} \sin \theta  \quad  e^{-i(\phi+\psi)} \cos \theta
\end{aligned}
\right),
\]
and their arbitrary mixtures. Namely, Alice and Bob are supposed to
choose probability densities $f_A(U)$ and $f_B(U)$ such that
\[
\int_{SU(2)} f_A(U) \, dU=1, \quad \int_{SU(2)} f_B(U) \, dU=1,
\]
where the integration is with respect to the Haar measure on $SU(2)$,
and the final state becomes
\begin{equation}
\label{eqmixedunitarystr}
\rho_{fin} =\int\int f_Af_B (U_A\otimes U_B)\rho_{in} (U_A\otimes U_B)^* dU_A dU_B.
\end{equation}

Diagonal elements of this matrix defines the probabilities of the
outcomes $|OO\rangle$, $|OT\rangle$, $|TO\rangle$, $|TT\rangle$. Hence,
introducing the payoff operators for Alice and Bob by
  \[
  \hat S_A=\al |OO\rangle \langle OO| +\be |TT\rangle \langle TT|
  +\ga (|OT\rangle \langle OT| + |TO\rangle \langle TO|),
  \]
  \[
   \hat S_B=\be |OO\rangle \langle OO| +\al |TT\rangle \langle TT|
  +\ga (|OT\rangle \langle OT| + |TO\rangle \langle TO|),
  \]
  it follows that the expected payoffs are
  \[
  \E S_A(f_A,f_B)={\tr} \, (\rho_{fin} \hat S_A)
  = \int\int f_Af_B S_A(U_A,U_B)\, dU_A dU_B,
  \]
  \[
  \E S_B(f_A,f_B)={\tr} \, (\rho_{fin} \hat S_B)
  = \int\int f_Af_B S_B(U_A,U_B)\, dU_A dU_B,
  \]
  where
  \[
  S_A(U_A,U_B)={\tr} [(U_A\otimes U_B)\rho_{in} (U_A\otimes U_B)^*\hat S_A],
  \]
  \[
  S_B(U_A,U_B)={\tr} [(U_A\otimes U_B)\rho_{in} (U_A\otimes U_B)^*\hat S_B]
  \]
  are their payoffs in the pure unitary strategies.
  Thus all the quantum content is encoded in the structure of these payoffs, and the problem
  to find Nash equilibria is now fully classical. The calculations can be simplified
  by noting that, since $\rho_{in}$ is decomposable density matrix,
  so are also the matrices $(U_A\otimes U_B)\rho_{in} (U_A\otimes U_B)^*$,
  and hence probabilities to get an outcome $|\si \tau\rangle$
  (here $\si$ and $\tau$ are either $O$ or $T$) in the game with pure strategies
  is just
  \[
  P_{\si \tau} =|\langle \si \tau |(U_A \otimes U_B)|\psi_{in}\rangle|^2,
  \]
  so that
  \[
  S_A(U_A,U_B)=\al |\langle OO |(U_A \otimes U_B)|\psi_{in}\rangle|^2
  +\be |\langle TT |(U_A \otimes U_B)|\psi_{in}\rangle|^2
  \]
  \[
  +\ga (|\langle OT |(U_A \otimes U_B)|\psi_{in}\rangle|^2
  + |\langle TO |(U_A \otimes U_B)|\psi_{in}\rangle|^2),
  \]
  and $S_B(U_A,U_B)$ the same with $\al,\be$ interchanged.

Calculations show (see \cite{DuLi01}) that in this game there are infinitely many Nash equilibria,
but they all give the same payoff $S=(\al+\be+2\ga)/4$. Moreover, the problem
of mismatch does not arise, because all Nash equilibria are of the form $(f_A, f_B)$
with $f_A\in F_A, f_B \in F_B$ and some sets $F_A, F_B$, and any combination yields
the same payoff.

In \cite{Frack13} another new version of MW protocol is introduced and analyzed, where
players are allowed additional choice, to accept given initial entangled state or not.
More precisely, both players declare independently whether they like to start with
a suggested (by referee) quantum state, and this quantum state is actually prepared
by the referee if both players declare
their willingness for it, otherwise they start with the classical initial $|OO\rangle$.

In \cite{Mendes05} the MW protocol and the transformation  \eqref{eqMWgenpaytr1} are used
for the quantization of the ultimatum game, where the first player is supposed
to have two strategies: to offer some preassigned unfair division of the total sum of $100$,
say $99+1$ and the fair one: $50+50$. The second player can either accept the offer or reject.
Thus the table is
 \[
\begin{tabular}{cc}
&Bob\\Alice&
\begin{tabular}{|c|c|c|}
\hline
&accept& not \\
\hline
unfair & (99,1) & (0,0)\\
\hline
fair & (50,50) & (0,0) \\
\hline
\end{tabular}
\\ &
\\ & {\bf Table 1.6}
\end{tabular}
\]

Paper \cite{Mendes05} also analyses this game under the set of all unitary strategies, where it just
reproduces for this concrete setting the general remark of \cite{Ben00}
on the absence of Nash equilibria for general MW protocol extended to full unitary strategies.
For mixed unitary strategies (like in \eqref{eqmixedunitarystr}) it is shown the existence of
Nash equilibria (the corresponding general result is given in Theorem \ref{thquantumNash}).

\section{Variations on EWL protocol}
\label{secVaronEWL}

As in the case of MW protocol, various extensions of EWL protocol were analyzed
by using more general strategy spaces and the games with more players
and more initial classical actions. Let us review some of these contributions.

In \cite{DuLi01a} the EWL protocol is applied to the Battle of Sexes. It is shown the existence of
infinitely many Nash equilibria when the strategies of players are restricted to
a two-parameter set of unitary transformations (like in the original EWL protocol). What seems more
important they show that for 'nontrivial' two action two player games, if the players are allowed to
play the full set of $SU(2)$ strategies, the quantum EWL game has no Nash equilibria, when started in
maximally entangled state  (see exact formulation below in Proposition \ref{noNash quantumtwo}).

In \cite{DuLi03} the analysis of equilibria for the general
 prisoner's dilemma with Table 1.3 above was provided. Under
 restricted set of unitary operators (EWL like), the phase
transitions are found:  the desired cooperative equilibria
$Q$ arises when the entanglement parameter $\ga$ crosses certain
critical values expressed in terms of the parameters $r,s,t$.
For the full unitary strategies there is a similar transition
between the situation with infinitely many equilibria and no
equilibria at all.

In \cite{DuLi03a}
the three player quantum Prisoner's dilemma is considered. There
are two natural equivalent ways to represent three player games,
via two tables distinguished by a particular choice of the third player:

\[
\begin{tabular}{cc}
&Bob\\Alice&
\begin{tabular}{|c|c|c|}
\hline
&C&D\\
\hline
C & (3,3,3) & (2,5,2)\\
\hline
D & (5,2,2) & (4,4,0) \\
\hline
\end{tabular}
\\ &
\\ & {\bf Table 1.6}: Colin C
\end{tabular}
\quad
\begin{tabular}{cc}
&Bob\\Alice&
\begin{tabular}{|c|c|c|}
\hline
&C&D\\
\hline
C & (2,2,5) & (0,4,4)\\
\hline
D & (4,0,4) & (1,1,1) \\
\hline
\end{tabular}
\\ &
\\ & {\bf Table 1.7}: Colin D
\end{tabular}
\]

or equivalently by a two-row table showing payoffs of each player $A$, $B$, $C$
obtained from each possible profile:
\[
\begin{tabular}{|c|c|c|c|c|c|c|c|c|}
\hline
profile & (C,C,C) & (C,D,C)& (D,C,C)& (D,D,C)& (C,C,D)& (C,D,D)& (D,C,D)& (D,D,D) \\
\hline
payoff & (3,3,3) & (2,5,2)& (5,2,2)& (4,4,0)& (2,2,5)& (0,4,4)& (4,0,4)& (1,1,1) \\
\hline
\end{tabular}
\]

The story behind the dilemma is the same as for two prisoners. The payoffs are chosen
 to reflect the idea that defection brings advantage to each player that is 'inversely
 proportional' to the number of other defecting players.

Since the game is symmetric, one can represent it also by a reduced table, where
all entries with the equal numbers of $D$ and $C$ are shown only once.
For the general payoffs of a three-player symmetric game the table can be given as

\begin{equation}
\label{eqtab3pris}
\begin{tabular}{|c|c|c|c|c|}
\hline
profile & (C,C,C) & (D,D,D)& (D,C,C)& (D,D,C) \\
\hline
payoff & ($r_{3c}, r_{3c},r_{3c}$) & ($r_{3d}, r_{3d},r_{3d}$)
& ($r_{1d}, r_{2c},r_{2c}$) & ($r_{2d}, r_{2d},r_{1c}$) \\
\hline
\end{tabular}
\end{equation}

The above story of the Prisoner's dilemma corresponds to the ordering
\[
r_{1c} < r_{3d} < r_{2c}< r_{3c} <r_{2d} < r_{1d}.
\]

Like its two-player counterpart, the classical version of this game is
a symmetric game with defecting $D$ being the dominating strategies,
so that the profile $(D,D,D)$ is a Nash equilibrium that is also the
solution in dominating strategies.

Quantum scheme extends the two-player game by choosing
$J=\exp\{i(\ga/2) \si_x \otimes \si_x \times \si_x\}$,
with $0\le \ga \le \pi/2$. The final state is
\[
|\psi_{fin} =J^*(U_A\otimes U_B\otimes U_C)J|OOO\rangle.
\]
The payoff for Alice, say, is
\[
S_A=5 P_{DCC}+4(P_{DDC}+P_{DCD}) +3P_{CCC} +2(P_{CCD}+P_{CDC})+P_{DDD},
\]
where
\[
P_{\si \xi \eta}=|\langle \si \xi \eta |\psi_{fin}\rangle |^2.
\]
The strategic space is chosen to be restricted to the two-parameter set:
\[
U(\theta, \phi)= \left(\begin{aligned}
& \,\, \cos \theta  \quad  \,\, e^{i\phi} \sin \theta \\
& -e^{-i\phi} \sin \theta  \quad  \cos \theta
\end{aligned}
\right), \quad \theta \in [0, \pi/2], \,\, \phi \in  [0, \pi/2].
\]
Here $U(0,0)=\1$ represents the strategy 'cooperate' and $U(\pi/2,\pi/2)=i\si_x$
represents the flipping operator of the 'defecting' strategy.

The calculations show (see \cite{DuLi03a}) that $i\si_y\otimes i\si_y \otimes i\si_y$
is a Nash equilibrium for all $\ga$ (this is a new feature as compared to the two-player
setting) with the payoff
\[
S_A=S_B=S_C=1+2\sin^2\ga.
\]
For $\ga=\pi/2$ this yields the desired cooperative and Pareto optimal payoffs of value $3$.
We see also that the symmetric equilibrium payoff increases monotonically and continuously with
the entanglement parameter $\ga$.

A straightforward extension of EWL scheme to arbitrary number $N$ of players is as follows.
The entangling operator is taken to be the 'maximally entangling' one:
\[
J=\frac{1}{\sqrt 2}(I^{\otimes N}+iF^{\otimes N})
=\frac{1}{\sqrt 2}(I^{\otimes N}+i\si_x^{\otimes N}),
\quad J^*=\frac{1}{\sqrt 2}(I^{\otimes N}-i\si_x^{\otimes N}),
\]
and thus the initial state is
\[
\xi_{in}=J(|0 \rangle^{\otimes N})=J|0\cdots 0\rangle=
\frac{1}{\sqrt 2}(|0 \rangle^{\otimes N}+i|1 \rangle^{\otimes N}).
\]
The final state is
\[
\xi_{fin}=J^* (U_1\otimes \cdots \otimes U_N) \xi_{in}
=J^* (U_1\otimes \cdots \otimes U_N)J|0 \rangle^{\otimes N}.
\]
Abandoning the artificial restrictions to the allowed unitary strategies, the game can
be naturally considered with arbitrary unitary strategies $U\in SU(2)$ of all players,
that is with $U$ given by \eqref{eqgenunit2dim5}:
\[
U=\cos \theta (\cos \phi \1 +i\sin \phi \si_z)
+i\sin \theta (\sin \psi \si_x+\cos\psi \si_y).
\]

Two further extensions of the strategy spaces are natural. One can use classically mixed
quantum strategies (like \eqref{eqmixedunitarystr} for MW protocols), or one can allow
to players to use the full set of $TP-CP$ operations (see Section \ref{secOpenSyst}),
their Kraus representation being given in \eqref{eq1ChoiKrausrep}.

\begin{remark}
These strategies can be realized physically via the interaction with additional quantum
systems, referred in this context to as ancillas (or ancillary qubits), see Theorem \ref{thStinespr2rep},
where $H_C$ is the ancilla. For this reason the authors of \cite{BH01a} point out that
in quantum setting all strategies can be considered as 'deterministic'.
\end{remark}

With the general unitary operators $U$ as above,
\[
U |0\rangle =\cos \theta \cos \phi |0\rangle  +i \cos \theta \sin \phi |0\rangle
+i\sin \theta \sin \psi |1\rangle-\sin \theta \cos \psi |1\rangle
\]
\[
=\cos \theta e^{i\phi}  |0\rangle -\sin \theta e^{-i\psi} |1\rangle,
\]
\[
U |1\rangle =\cos \theta \cos \phi |1\rangle  -i \cos \theta \sin \phi |1\rangle
+i\sin \theta \sin \psi |0\rangle+\sin \theta \cos \psi |0\rangle
\]
\[
=\cos \theta e^{-i\phi}  |1\rangle +\sin \theta e^{i\psi} |0\rangle,
\]
and
\[
(U_1\otimes \cdots \otimes U_N)J|0 \rangle^{\otimes N}
=\frac{1}{\sqrt 2} U_1|0\rangle \otimes \cdots \otimes U_N|0\rangle
+  \frac{i}{\sqrt 2} U_1|1\rangle \otimes \cdots \otimes U_N|1\rangle.
\]

For instance, in the case $N=2$,
\[
U_1\otimes U_2 \xi_{in}
=\frac{1}{\sqrt 2}(\cos \theta_1 e^{i\phi_1}  |0\rangle -\sin \theta_1 e^{-i\psi_1} |1\rangle)
\otimes (\cos \theta_2 e^{i\phi_2}  |0\rangle -\sin \theta_2 e^{-i\psi_2} |1\rangle)
\]
\[
+ \frac{i}{\sqrt 2}(\cos \theta_1 e^{-i\phi_1}  |1\rangle +\sin \theta_1 e^{i\psi_1} |0\rangle)
\otimes (\cos \theta_2 e^{-i\phi_2}  |1\rangle +\sin \theta_2 e^{i\psi_2} |0\rangle).
\]

\[
\xi_{fin}=\frac{1}{2}(\cos \theta_1 e^{i\phi_1}  |0\rangle -\sin \theta_1 e^{-i\psi_1} |1\rangle)
\otimes (\cos \theta_2 e^{i\phi_2}  |0\rangle -\sin \theta_2 e^{-i\psi_2} |1\rangle)
\]
\[
+ \frac{i}{2}(\cos \theta_1 e^{-i\phi_1}  |1\rangle +\sin \theta_1 e^{i\psi_1} |0\rangle)
\otimes (\cos \theta_2 e^{-i\phi_2}  |1\rangle +\sin \theta_2 e^{i\psi_2} |0\rangle).
\]
\[
-\frac{i}{2}(\cos \theta_1 e^{i\phi_1}  |1\rangle -\sin \theta_1 e^{-i\psi_1} |0\rangle)
\otimes (\cos \theta_2 e^{i\phi_2}  |1\rangle -\sin \theta_2 e^{-i\psi_2} |0\rangle)
\]
\[
+ \frac{1}{2}(\cos \theta_1 e^{-i\phi_1}  |0\rangle +\sin \theta_1 e^{i\psi_1} |1\rangle)
\otimes (\cos \theta_2 e^{-i\phi_2}  |0\rangle +\sin \theta_2 e^{i\psi_2} |1\rangle)
\]
\[
=\sum_{k,l=0}^1 \xi_{kl} |kl\rangle,
\]
 \[
\xi_{00}=\frac12 \cos \theta_1 e^{i\phi_1}\cos \theta_2 e^{i\phi_2}
+ \frac12 i \sin \theta_1 e^{i\psi_1}\sin \theta_2 e^{i\psi_2}
\]
\[
- \frac12 i \sin \theta_1 e^{-i\psi_1}\sin \theta_2 e^{-i\psi_2}
+\cos \theta_1 e^{-i\phi_1} \cos \theta_2 e^{-i\phi_2}
\]
 \begin{equation}
\label{eqprobout2play2ac1}
=\cos \theta_1 \cos \theta_2 \cos(\phi_1+\phi_2)
-\sin \theta_1 \sin \theta_2 \sin(\psi_1+\psi_2),
\end{equation}
\[
\xi_{11}=\frac12 \sin \theta_1 e^{-i\psi_1}\sin \theta_2 e^{-i\psi_2}
+\frac12 i\cos \theta_1 e^{-i\phi_1}\cos \theta_2 e^{-i\phi_2}
\]
\[
-\frac12 i\cos \theta_1 e^{i\phi_1} \cos \theta_2 e^{i\phi_2}
+\frac12 \sin \theta_1 e^{i\psi_1}\sin \theta_2 e^{i\psi_2}
\]
 \begin{equation}
\label{eqprobout2play2ac2}
= \sin \theta_1 \sin \theta_2 \cos (\psi_1+\psi_2)
+\cos \theta_1 \cos \theta_2 \sin (\phi_1+\phi_2),
\end{equation}
\[
\xi_{01}=-\frac12\cos \theta_1 e^{i\phi_1}  \sin \theta_2 e^{-i\psi_2}
+\frac{i}{2} \sin \theta_1 e^{i\psi_1} \cos \theta_2 e^{-i\phi_2}
\]
\[
 \frac{i}{2}\sin \theta_1 e^{-i\psi_1} \cos \theta_2 e^{i\phi_2}
 \frac{1}{2} \cos \theta_1 e^{-i\phi_1}\sin \theta_2 e^{i\psi_2}
\]
 \begin{equation}
\label{eqprobout2play2ac3}
=i \cos \theta_1  \sin \theta_2 \sin (\psi_2-\phi_1)
+i \sin \theta_1  \cos \theta_2 \cos (\phi_2-\psi_1),
\end{equation}
 \begin{equation}
\label{eqprobout2play2ac4}
\xi_{10}=i \cos \theta_2  \sin \theta_1 \sin (\psi_1-\phi_2)
+i \sin \theta_2  \cos \theta_1 \cos (\phi_1-\psi_2).
\end{equation}

For the probabilities of the four outcomes we thus have
\[
P_{00}=|\xi_{00}|^2
=\cos^2 \theta_1 \cos^2 \theta_2 \cos^2(\phi_1+\phi_2)
+\sin^2 \theta_1 \sin^2 \theta_2 \sin^2(\psi_1+\psi_2)
\]
\[
-2\sin \theta_1 \sin \theta_2 \cos \theta_1 \cos \theta_2 \cos(\phi_1+\phi_2) \sin(\psi_1+\psi_2),
\]
\[
P_{11} =|\xi_{11}|^2
=\cos^2 \theta_1 \cos^2 \theta_2 \sin^2(\phi_1+\phi_2)
+\sin^2 \theta_1 \sin^2 \theta_2 \cos^2(\psi_1+\psi_2)
\]
\[
+2\sin \theta_1 \sin \theta_2 \cos \theta_1 \cos \theta_2 \cos(\psi_1+\psi_2) \sin(\phi_1+\phi_2),
\]
\[
P_{01}=|\xi_{01}|^2
=\cos^2 \theta_1  \sin^2 \theta_2 \sin^2 (\psi_2-\phi_1)
+ \sin^2 \theta_1  \cos^2 \theta_2 \cos^2 (\phi_2-\psi_1),
\]
\[
+2 \sin \theta_1 \sin \theta_2 \cos \theta_1  \cos \theta_2
\sin (\psi_2-\phi_1) \cos (\phi_2-\psi_1),
\]
\[
P_{10}=|\xi_{10}|^2
=\cos^2 \theta_2  \sin^2 \theta_1 \sin^2 (\psi_1-\phi_2)
+ \sin^2 \theta_2  \cos^2 \theta_1 \cos^2 (\phi_1-\psi_2)
\]
\[
+2 \sin \theta_1 \sin \theta_2 \cos \theta_1  \cos \theta_2
\sin (\psi_1-\phi_2) \cos (\phi_1-\psi_2).
\]

\begin{exer}
Check that these probabilities really sum up to 1.
\end{exer}

\begin{prop}
\label{noNash quantumtwo}
Unless there is an outcome $|\si \tau\rangle$ such that it gives the best payoff to both
Alice and Bob, there is no Nash equilibrium for full pure quantum strategies $SU(2)$ for
a two-player two-action game.
\end{prop}

\begin{proof}
It is seen from formulas \eqref{eqprobout2play2ac1} - \eqref{eqprobout2play2ac4}
that whatever choice of parameter $\theta_2, \phi_2, \psi_2$ is made by Bob,
Alice can choose her $\theta_1, \phi_1, \psi_1$ in a way that would make any of the
coefficients $\xi_{00}$, $\xi_{01}$, $\xi_{10}$, $\xi_{11}$  equal $1$ in magnitude,
and thus to ensure the corresponding outcome to occur with probability $1$. For instance,
in order to achieve $\xi_{00}=1$, Alice can choose $\phi_1=-\phi_2$, $\psi_1=(\pi/2)-\psi_2$,
which turns $\xi_{00}$ to
\[
\cos \theta_1 \cos \theta_2 -\sin \theta_1 \sin \theta_2 =\cos (\theta_1+\theta_2)
\]
and then $\theta_1=-\theta_2$ converts this to $1$. The same is possible
for Bob under any strategy of Alice. Hence a Nash equilibrium can be only
an outcome that  gives the best payoff to both Alice and Bob.
\end{proof}

Remarkably enough, for $N>2$ the situation changes drastically.
For $N=3$ we have
\[
U_1\otimes U_2\otimes U_3 \xi_{in}
=\frac{1}{\sqrt 2}(\cos \theta_1 e^{i\phi_1}  |0\rangle -\sin \theta_1 e^{-i\psi_1} |1\rangle)
\otimes (\cdots {}_2 \cdots) \otimes (\cdots {}_3 \cdots)
\]
\[
+ \frac{i}{\sqrt 2}(\cos \theta_1 e^{-i\phi_1}  |1\rangle +\sin \theta_1 e^{i\psi_1} |0\rangle)
\otimes (\cdots {}_2 \cdots) \otimes (\cdots {}_3 \cdots).
\]

\[
\xi_{fin}=
\frac{1}{2}(\cos \theta_1 e^{i\phi_1}  |0\rangle -\sin \theta_1 e^{-i\psi_1} |1\rangle)
\otimes (\cdots {}_2 \cdots) \otimes (\cdots {}_3 \cdots)
\]
\[
+ \frac{i}{2}(\cos \theta_1 e^{-i\phi_1}  |1\rangle +\sin \theta_1 e^{i\psi_1} |0\rangle)
\otimes (\cdots {}_2 \cdots) \otimes (\cdots {}_3 \cdots)
\]
\[
-\frac{i}{2}(\cos \theta_1 e^{i\phi_1}  |1\rangle -\sin \theta_1 e^{-i\psi_1} |0\rangle)
\otimes (\cdots {}_2 \cdots) \otimes (\cdots {}_3 \cdots)
\]
\[
+ \frac{1}{2}(\cos \theta_1 e^{-i\phi_1}  |0\rangle +\sin \theta_1 e^{i\psi_1} |1\rangle)
\otimes (\cdots {}_2 \cdots) \otimes (\cdots {}_3 \cdots)
\]
\[
=\sum_{k,l,m=0}^1 \xi_{klm} |klm\rangle,
\]
where the second and the third brackets reproduce the first one with all indices changed
 to $2$ or $3$ respectively.
This is the linear combination of the $8$ basis vectors. But only $4$ need to be calculated,
as the other are obtained by permutations. We have
\[
\xi_{000}=\frac12 \cos \theta_1 e^{i\phi_1}\cos \theta_2 e^{i\phi_2}\cos \theta_3 e^{i\phi_3}
+\frac12 i\sin \theta_1 e^{i\psi_1} \sin \theta_2 e^{i\psi_2} \sin \theta_3 e^{i\psi_3}
\]
\[
 +\frac12 i\sin \theta_1 e^{-i\psi_1}\sin \theta_2 e^{-i\psi_2}\sin \theta_3 e^{-i\psi_3}
 +\frac12 \cos \theta_1 e^{-i\phi_1} \cos \theta_2 e^{-i\phi_2}\cos \theta_3 e^{-i\phi_3}
 \]
 \[
 =  \cos \theta_1 \cos \theta_2 \cos \theta_3 \cos (\phi_1+\phi_2 +\phi_3)
 + i \sin \theta_1 \sin \theta_2 \sin \theta_3 \cos (\psi_1+\psi_2 +\psi_3),
 \]
 \[
 \xi_{111}=-\frac12 \sin \theta_1 e^{-i\psi_1}\sin \theta_2 e^{-i\psi_2}\sin \theta_3 e^{-i\psi_3}
+\frac12 i \cos \theta_1 e^{-i\phi_1} \cos \theta_2 e^{-i\phi_2} \cos \theta_3 e^{-i\phi_3}
\]
\[
-\frac12 i \cos \theta_1 e^{i\phi_1}\cos \theta_2 e^{i\phi_2}\cos \theta_3 e^{i\phi_3}
+\frac12 \sin \theta_1 e^{i\psi_1}\sin \theta_2 e^{i\psi_2}\sin \theta_3 e^{i\psi_3}
\]
\[
=i\sin \theta_1 \sin \theta_2 \sin \theta_3 \sin (\psi_1+\psi_2+\psi_3)
+\cos \theta_1 \cos \theta_2 \cos \theta_3  \sin (\phi_1+\phi_2+\phi_3),
\]
\[
\xi_{001} = -\frac12 \cos \theta_1 e^{i\phi_1}\cos \theta_2  e^{i\phi_2} \sin \theta_3 e^{-i\psi_3}
+ \frac12 i \sin \theta_1 e^{i\psi_1} \sin \theta_2 e^{i\psi_2} \cos \theta_3 e^{-i\phi_3}
\]
\[
-\frac12 i \sin \theta_1 e^{-i\psi_1}\sin \theta_2 e^{-i\psi_2}\cos \theta_3 e^{i\phi_3}
  +\frac12 \cos \theta_1 e^{-i\phi_1} \cos \theta_2 e^{-i\phi_2} \sin \theta_3 e^{i\psi_3}
  \]
  \[
  = i\cos \theta_1 \cos \theta_2  \sin \theta_3 \sin (\psi_3-\phi_1-\phi_2)
+  \sin \theta_1 \sin \theta_2 \cos \theta_3 \sin (\phi_3-\psi_1-\psi_2),
  \]
  \[
  \xi_{011}=\frac12 \cos \theta_1 e^{i\phi_1} \sin \theta_2 e^{-i\psi_2}\sin \theta_3 e^{-i\psi_3}
  +\frac12 i \sin \theta_1 e^{i\psi_1} \cos \theta_2 e^{-i\phi_2} \cos \theta_3 e^{-i\phi_3}
\]
\[
+\frac12 i \sin \theta_1 e^{-i\psi_1} \cos \theta_2 e^{i\phi_2} \cos \theta_3 e^{i\phi_3}
+\frac12 \cos \theta_1 e^{-i\phi_1} \sin \theta_2 e^{i\psi_2} \sin \theta_3 e^{i\psi_3}
\]
\[
= \cos \theta_1  \sin \theta_2 \sin \theta_3 \cos (\psi_2+\psi_3-\phi_1)
+ i \sin \theta_1 \cos \theta_2  \cos \theta_3 \cos(\phi_3+\phi_2-\psi_1).
 \]

In \cite{BH01a} a detailed discussion is devoted to the quantized version of the famous minority game.
In its classical versions the players are supposed to choose $0$  or $1$ and submit to the referee.
Those whose choice turns out to be in minority get one point reward each. If there is an even split,
or all player made the same choice, no payments arise.

 Probability for 1st player to be in minority is
 \[
 |\xi_{011}|^2 +|\xi_{100}|^2
 \]
 \[
 =\cos^2\theta_1\sin^2\theta_2\sin^2\theta_3\cos^2(\phi_1-\psi_2-\psi_3)
 +\sin^2\theta_1 \cos^2\theta_2\cos^2\theta_3\cos^2(\psi_1-\phi_2-\phi_3)
 \]
 \[
 +\cos^2\theta_1\sin^2\theta_2\sin^2\theta_3\sin^2(\phi_1-\psi_2-\psi_3)
 +\sin^2\theta_1\cos^2\theta_2\cos^2\theta_3\sin^2(\psi_1-\phi_2-\phi_3)
 \]
 \[
=\cos^2\theta_1\sin^2\theta_2\sin^2\theta_3+\sin^2\theta_1\cos^2\theta_2\cos^2\theta_3,
 \]
 which is the same as in the classical game for $\cos^2 \theta=p$ denoting the probability of flipping
 (or of choosing $1$). Hence for $N=3$ player the quantum version of the minority game
 does not offer anything new. T

 The situation changes when the number of payers increases. The analysis of these cases exploits a simple
 observation that for minority games the result is not changed whether or not the final
 gate $J^*$ is applied. In fact, $J^*$ transforms any basis vectors $j_1 \cdots j_k\rangle$
 within the sub-space generated by $j_1 \cdots j_k\rangle$ and $j'_1 \cdots j'_k\rangle$
 (prime denotes the complimentary index), but both these vectors yield the same payoff.
 Hence for these particular games EWL and MW schemes are equivalent. It is shown
 in \cite{BH01a} that new (and more profitable than classical) equilibria arise
 for the minority games with $N>3$. An example of such equilibrium for $N=4$ is $(u,u,u,u)$,
 where
 \[
 u= \frac{1}{\sqrt 2} \cos \frac{\pi}{16}(I+i\si_x)+\frac{1}{\sqrt 2} \sin \frac{\pi}{16}(i\si_y-i\si_z).
 \]

 As also shown in \cite{BH01a}, there exist games of $3$ player where new profitable equilibria arise.
 For instance, the game with the table of type  \eqref{eqtab3pris}:
  \[
\begin{tabular}{|c|c|c|c|c|}
\hline
profile & (C,C,C) & (D,D,D)& (D,C,C)& (D,D,C) \\
\hline
payoff & (2,2,2) & (0,0,0) & (1,9,9) & (-9,-9,1) \\
\hline
\end{tabular}
\]
have new profitable equilibria in its quantum version. The equilibria are given
by unitary strategies, but represent equilibria even if considered among all
$TP-CP$ strategies (which is proved using the Kraus representations for such maps).
On the other hand, there are examples, for instance given by the table
\[
\begin{tabular}{|c|c|c|c|c|}
\hline
profile & (C,C,C) & (D,D,D)& (D,C,C)& (D,D,C) \\
\hline
payoff & (-9,-9,-9) & (7,7,7) & (8,-9,-9) & (1,1,-9) \\
\hline
\end{tabular}
\]
where classical rules produce outcomes (Nash equilibria) with better performance
than their quantum counterparts.

\section{Quantization of games with continuous strategy spaces}
\label{secQGcontstate}

The extension of EWL protocol for games with initially continuous strategy space
was first suggested in \cite{LiContvarquantgame02}.

The underlying classical model was that of Cournot's duopoly.
Recall that, for $Q=q_1+q_2 $ denoting the total amount of a product produced by two firms,
one assumes that the price per unit of the product equals $P(Q)=(a-Q)^+=\max(0, a-Q)$.
If $c$ is the cost of the production of a unit of the product, the profits of two firms are
(for $a\ge Q$)
\begin{equation}
\label{eqCournotprepqu}
u_j(q_1,q_2)=q_j [P(Q)-c]=q_j[a-c -(q_1+q_2)]
\end{equation}
Though the unique Nash equilibrium is $q_1^*=q_2^*=(a-c)/3$ with each firm getting
$(a-c)^2/9$, the cooperative behavior would be to choose $q'_1=q'_2=(a-c)/4$ yielding
to each firm the better profit $(a-c)^2/4$.

To quantise this game let us assume that each player is working with the Hilbert space $L^2(\R)$.
The simplest initial functions for both players are the Gaussian packets
\[
\psi_j(x_j) = (\pi h)^{-1/4} \exp\left\{-\frac{x_j^2}{2h}\right\}
\]
(normalized to $\int |\psi_j(x_j)|^2 dx_j=1$, so that
\begin{equation}
\label{eqQuContSpace0}
\psi^{in}(x_1,x_2)=\psi_1(x_1) \psi_2(x_2)=(\pi h)^{-1/2} \exp\{-\frac{x_1^2+x_2^2}{2h}\}.
\end{equation}

The two basic operators in $L^2(\R)$ are the operator $X$ of multiplication
by the variable $x$ and the momentum operator $P=-ihd/dx$.
The unitary shift operators
\[
D(y)f(x)=\exp\{-iyP/h\}f(x)=\exp\{-yd/dx\}f(x)=f(x-y)
\]
are the simplest possible operators allowing the players to manipulate their positions
(the amount of product to produce). Therefore they are natural candidates to be chosen
as possible actions of the players. Thus, copying the finite-dimensional EWL scheme,
we can introduce a quantum  version of Cournot's game by asserting that the final
state of the system should be
\begin{equation}
\label{eqQuContSpace1}
\psi^{fin}_{y_1,y_2}=J^* [D_1(y_1) \otimes D_2(y_2)]J \psi_{in}
\end{equation}
with an appropriately chosen unitary entangling operator $J$ on
$L^2(\R)\otimes L^2(\R)=L^2(\R^2)$. By the canonical interpretation
of the wave mechanics, the probability distribution of finding a system
described by the wave function $\psi(x_1,x_2)$ in a position $(x_1,x_2)$
has the probability density $|\psi(x_1,x_2)|^2$. Hence the average positions
and final payoffs to the players
can be calculated by the formulas
\begin{equation}
\label{eqQuContSpace2}
q_j (y_1,y_2)=\E [x_j] =\int\int x_j |\psi^{fin}_{y_1,y_2}(x_1,x_2)|^2 \, dx_1 dx_2,
\end{equation}
\begin{equation}
\label{eqQuContSpace3}
u_j(y_1,y_2)=\E [x_j(a-c -(x_1+x_2))]
=\int\int [x_j(a-c -(x_1+x_2))] |\psi^{fin}_{y_1,y_2}(x_1,x_2)|^2 \, dx_1 dx_2.
\end{equation}

Looking at the simplest $J$ that may mix up the variables, one can copy the
unitary rotations  of \eqref{eqgenunit2dim3a} and suggest to use $J$ of the type
\begin{equation}
\label{eqQuContSpace4}
Jf(x_1,x_2)=f(U(x_1,x_2)), \quad U(x_1,x_2) =\left(\begin{aligned}
& \quad \cos c \quad \sin c  \\
& -\sin c \quad \cos c
\end{aligned}
\right)\left(\begin{aligned}
& x_1  \\
& x_2
\end{aligned}
\right).
\end{equation}

However, by physical reasons (see some comments below) the authors of
\cite{LoStackelberg03} suggest to use instead the 'Lorenz rotations':
\[
J_{\ga}f(x_1,x_2)=f(U_{\ga}(x_1,x_2)),
\]
\begin{equation}
\label{eqQuContSpace5}
 U_{\ga}(x_1,x_2) =\left(\begin{aligned}
& \cosh \ga \quad \sinh \ga  \\
& \sinh \ga \quad \cosh \ga
\end{aligned}
\right)\left(\begin{aligned}
& x_1  \\
& x_2
\end{aligned}
\right)
=\left(\begin{aligned}
& x_1 \cosh \ga  + x_2 \sinh \ga  \\
& x_1 \sinh \ga x_1 + x_2 \cosh \ga
\end{aligned}
\right),
\end{equation}
with the inverse operator
 \begin{equation}
\label{eqQuContSpace6}
J^*_{\ga}f(x_1,x_2)=J^{-1}_{\ga}f(x_1,x_2)=f(U^{-1}_{\ga}(x_1,x_2)),
\quad U^{-1}_{\ga}(x_1,x_2) =\left(\begin{aligned}
& \quad \cosh \ga \quad -\sinh \ga  \\
& -\sinh \ga \quad \quad \cosh \ga
\end{aligned}
\right)\left(\begin{aligned}
& x_1  \\
& x_2
\end{aligned}
\right).
\end{equation}

With this choice of $J_{\ga}$  and denoting $x=(x_1,x_2)$, $y=(y_1,y_2)$,  we get for
an arbitrary  $\psi(x_1,x_2)$ that
\[
(J \psi)(x_1,x_2)=(J_{\ga} \psi)(x)=\psi(Ux),
\]
\[
[D_1(y_1) \otimes D_2(y_2)]J_{\ga} \psi (x_1,x_2)=\psi(Ux-Uy),
\]
\[
J^*_{\ga}[D_1(y_1) \otimes D_2(y_2)]J_{\ga} \psi (x_1,x_2)=\psi(x-Uy).
\]
Therefore, with $\psi^{in}$ given by \eqref{eqQuContSpace0},
 \[
\psi^{fin}_{y_1,y_2}=J^*_{\ga} [D_1(y_1) \otimes D_2(y_2)]J_{\ga} \psi_{in}
\]
 \begin{equation}
\label{eqQuContSpace7}
=(\pi h)^{-1/2} \exp\{-\frac{1}{2h} [(x_1-y_1 \cosh \ga  - y_2 \sinh \ga )^2
+(x_2-y_1 \sinh \ga + y_2 \cosh \ga )^2]\}.
\end{equation}

The average positions in this state, defined by \eqref{eqQuContSpace2}, equal
\begin{equation}
\label{eqQuContSpace8}
q_1 (y_1,y_2)=y_1 \cosh \ga  + y_2 \sinh \ga , \quad q_1 (y_1,y_2)=y_1 \sinh \ga  + y_2 \cosh \ga .
\end{equation}

Though in principle we are mostly interested in payoffs \eqref{eqQuContSpace3},
the final simplification suggested in  \cite{LiContvarquantgame02} is that before the measurement,
the final state $\psi^{fin}_{y_1,y_2}$ is squeezed in a way that it effectively becomes the
$\de$-function centered at the mean position $(q_1,q_2)$, and therefore, instead of  \eqref{eqQuContSpace3},
the payoffs simplify to \eqref{eqCournotprepqu} with $(q_1,q_2)$ given by \eqref{eqQuContSpace8}:
\begin{equation}
\label{eqQuContSpace9}
\begin{aligned}
& u_1^{\ga}(y_1,y_2)=q_1[a-c -(q_1+q_2)]=(y_1 \cosh \ga  + y_2 \sinh \ga )[a-c-e^{\ga}(y_1+y_2)], \\
& u_2^{\ga}(y_1,y_2)=q_2[a-c -(q_1+q_2)]=(y_1 \sinh \ga  + y_2\cosh \ga )[a-c-e^{\ga}(y_1+y_2)].
\end{aligned}
\end{equation}

Solving for the Nash equilibrium, that is, solving the equations
\begin{equation}
\label{eqQuContSpace10a}
\frac{\pa u_1}{\pa y_1}=0, \quad   \frac{\pa u_2}{\pa y_2}=0,
\end{equation}
yields the equilibrium
\begin{equation}
\label{eqQuContSpace10}
y_1^*=y_2^*=\frac {(a-c) \cosh \ga}{1+2e^{2\ga}}
\end{equation}
with the profit
\begin{equation}
\label{eqQuContSpace11}
u_1^{\ga}(y_1^*,y_2^*) = u_2^{\ga}(y_1^*,y_2^*)=\frac {(a-c)^2 e^{\ga} \cosh \ga}{(3\cosh \ga+ \sinh \ga)^2}.
\end{equation}

As $\ga=0$ we recover the classical game. But as $\ga \to \infty$, we have
\begin{equation}
\label{eqQuContSpace11a}
\lim_{\ga\to \infty} u_j^{\ga}(y_1^*,y_2^*) = (a-c)^2/8,
\end{equation}
which is the  effective outcome. Thus in this limit the dilemma between the Nash equilibrium
and the Pareto optimum disappears.

\begin{exer} Calculate the Nash equilibrium using the full formula \eqref{eqQuContSpace3}
instead of its simplified version \eqref{eqQuContSpace9}.
\end{exer}

\begin{remark} Physical realization of quantum games are usually performed via the methods of quantum optics.
There the main role is played by the creation and annihilation operators $\hat a^{\pm}$ of quantum oscillators,
which are given by the formulas
\[
\hat a^{\pm}=\frac{1}{\sqrt{2\om h}}(\om X \mp iP),
\]
or equivalently
\[
X=\sqrt{\frac{\hbar}{2\om}} (\hat a^- +\hat a^+), \quad P=-i \sqrt{\frac{\hbar \om}{2}} (\hat a^- -\hat a^+).
\]
In quantum optics the operators $X$ and $P$ are referred to as the quadratures  (of a single mode
of the electromagnetic field given by $\hat a^{\pm}$). In paper   \cite{LiContvarquantgame02}
the units with $\om=1$ and $h=1$ are used, in which case it is seen that  the operator $J$
of \eqref{eqQuContSpace5} is given by the formula
\begin{equation}
\label{eqQuContSpace12}
J(\ga) =\exp\{-\ga (\hat a_1^+\hat a_2 ^+-\hat a_1^-\hat a_2 ^-) \}
=\exp\{i\ga (X_1P_2+X_2P_1)\},
\end{equation}
and $J_{\ga}\psi_{in}$ turns out to represent the important two-mode squeezed vacuum state used in
the theory of quantum teleportation.
\end{remark}

In \cite{LoStackelberg03} the above scheme (again with the simplification \eqref{eqQuContSpace9})
was used to analyze the Stackelberg duopoly.
The difference with the above game is that now the moves are sequential. Firstly the first firm
makes the move by choosing $y_1$, and then the second firm makes the move choosing its $y_2$
that maximises its profit given $y_1$. Thus the optimal choice of the second firm
arises from solving the second equation in \eqref{eqQuContSpace10a} yielding
\[
y_2(y_1)=\frac{(a-c) \cosh \ga -y_1 e^{2\ga}}{1+e^{2\ga}}
\]
Then the first firm should find $y_1$ maximising
\[
u_1^{\ga}(y_1,y_2(y_1))=(y_1 \cosh \ga  + y_2(y_1) \sinh \ga )[a-c-e^{\ga}(y_1+y_2(y_1))].
\]
Simple analysis yields the optimal value
\[
 y_1^*=\frac{(a-c)(1+\cosh (2\ga)}{2(\cosh \ga +e^{\ga})},
 \]
 with the corresponding optimal $y_2^*=y_2(y_1^*)$. Of course, the optimal profit
 of the second firm turns out to be lower than the optimal profit of the first firm
 (advantage of the first move). Moreover, the difference between the optimal profit
 of the two firms is a monotonically increasing function with respect to
 the 'entangling parameter' $\ga$.

In paper \cite{ZhouMulytiplayerQuantContStrat05} the above results were
extended to the case of several firms. The arguments and results are mostly
analogous (the calculations being of course heavier). Let us notice only that
the operator $J$ of \eqref{eqQuContSpace5}
or \eqref{eqQuContSpace12} is generalized to the operator
\[
J=\exp\{ -\sum_{i\neq j} \ga_j  (\hat a_1^+\hat a_2 ^+-\hat a_1^-\hat a_2 ^-)\}.
\]

\section{Finite-dimensional quantum mechanics of open systems}
\label{secOpenSyst}

The transformations of open quantum systems may be performed by more
general operators than unitary. Namely, one defines {\it operations}
between the state spaces $\TC_s(H_A)$ and $\TC_s(H_B)$ as positive
linear maps $\TC_s(H_A)\to\TC_s(H_B)$ (that take positive linear
operators to positive linear operators), which are contractions in the trace norm:
\begin{equation}
\label{eqcontrtrace}
0\le {\tr} [T(\rho)] \le {\tr} (\rho)
\end{equation}
 for any $\rho \in \TC^+(H_A)$.
Since $|T(\rho)|=T(|\rho|)$ for a positivity preserving $T$ inequality \eqref{eqcontrtrace} is
equivalent to the inequality
\begin{equation}
\label{eqcontrtrace1}
0\le {\tr} |T(\rho)| \le {\tr} |\rho|
\end{equation}
for any $\rho \in \TC_s(H_A)$.

\begin{remark} Some authors define operations as CP-maps introduced below.
\end{remark}

Applying duality \eqref{dualopernorm},
for any $T\in \LC(\TC_s(H_A),\TC_s(H_B))$ one can define the dual map
$T^*: \LC (\LC_s(H_B),\LC_s(H_A))$ via the equation
\begin{equation}
\label{eqdualoper}
{\tr} [T(\rho) \si]={\tr} [\rho T^*(\si)]
\end{equation}

If $T$ is positive, then $T^*$ is also positive (as follows from \eqref{eqdualoper}).
Contraction property \eqref{eqcontrtrace} is equivalent to $T^*\1 \le \1$, and the preservation of the trace
by $T$ is equivalent to the preservation of unity by $T^*$: $T^*(\1)=\1$.

The following simple result is crucial for the theory of games.

\begin{lemma}
\label{lempositivearecompact}
Positive contractions preserving or not increasing trace (or preserving or not increasing the unity operator)
form a convex compact set in $\LC (\TC_s(H_A),\TC_s(H_B))$.
\end{lemma}

\begin{proof} It is straightforward to see that any of the 4 sets mentioned
are convex and closed in  $\LC (\TC_s(H_A),\TC_s(H_B))$. The only thing to check
for compactness is thus the boundedness, and it follows from \eqref{eqcontrtrace1}.
\end{proof}

This duality allows for the most straightforward method to introduce
the important notion of the partial trace. Namely,
by duality \eqref{eqdualoper}, if $\rho$ is a state on $\tilde H$,
the positive linear map $\si \mapsto \si\otimes \rho$ from $\TC(H) \to \TC(H\otimes \tilde H)$
has the adjoint positive linear map $E_{\rho}: \LC(H\otimes \tilde H) \to \LC(H) $,
called the {\it partial trace}\index{partial trace}. This mapping  satisfies the equation
\begin{equation}
\label{eqdefparrtrace}
{\tr} [E_{\rho} (A) \si]={\tr} [A(\si \otimes \rho)],
\end{equation}
with $A\in \LC(H\otimes \tilde H), \si \in \TC(H)$. Moreover,
\begin{equation}
\label{eqdefparrtrace1}
E_{\rho}(B\otimes D)=B \, {\tr} (D\rho ),
\end{equation}
because
\[
{\tr} [E_{\rho} (B\otimes D) \si]={\tr} [B\otimes D(\si \otimes \rho)]
={\tr} [B\si] {\tr} [D\rho].
\]

Since any operator in $\LC(H\otimes \tilde H)$ is a linear combinations
of the product operators of type $B\otimes D$, formula \eqref{eqdefparrtrace1}
can be taken as an equivalent definition of the partial trace.

In particular, if $\rho=\1$, formula \eqref{eqdefparrtrace1} reduces to
   \begin{equation}
\label{eqdefparrtrace2}
{\tr}_{\tilde H}(B\otimes D)= E_{\1}(B\otimes D)=B \, {\tr} D,
 \end{equation}
 the new left notation being seemingly the most commonly used one.

For example, if $\psi=(|00\rangle+|11\rangle)/\sqrt 2\in H_1\otimes H_2 =\C^2 \otimes \C^2$,
the corresponding density matrix is
\[
\rho=|\psi\rangle \langle \psi|=\frac12\sum_{j,k=0}^1|j\rangle \langle k|\otimes |j\rangle \langle k|
\]
and its partial trace is
\[
{\tr}_{H_2}\rho={\tr}_{H_1}\rho =\frac12 \sum_{j=0}^1|j\rangle \langle j|
=\left(
\begin{aligned}
& 1/2 \quad \,\, 0 \\
& \,\, 0 \quad 1/2
\end{aligned}
\right).
\]

Important fact is that any state $\rho$ can be written as a partial trace of
a pure state, called a {\it purification} of $\rho$.
In fact, for any state $\rho$ in $H$, in the basis $|\xi_j\rangle$, where $\rho$ is diagonal,
it can be written as $\rho=\sum_j \rho_j |\xi_j \rangle \langle \xi_j|$,
and a possible choice of pure state is $|\psi\rangle \langle \psi |$ with
\[
|\psi\rangle = \sum_j \sqrt {\rho_j} |\xi_j \rangle \otimes |\eta_j \rangle
\]
in $H\otimes H$, where $\eta_j$ is any orthonormal basis (for instance,$\eta_j=\xi_j$,
or $\eta_j =\bar\xi_j$).
More precisely, $|\psi\rangle$ can be chosen to lie in $H\otimes \tilde H$, where
the dimension of $\tilde H$ equals the rank of $\rho$ (the number of non-vanishing $\rho_j$).
Then
\[
|\psi\rangle \langle \psi |
=\sum_{j,k} \sqrt {\rho_j \rho_k} (|\xi_j \rangle \otimes |\eta_j \rangle)(\langle \xi_k| \otimes \langle \eta_k|)
=\sum_{j,k} \sqrt {\rho_j \rho_k} |\xi_j \rangle \langle \xi_k| \otimes  |\eta_j \rangle \langle \eta_k|.
\]
Taking partial trace only terms with $j=k$ survive, because ${\tr}  |\eta_j \rangle \langle \eta_k|=\de_j^k$ yielding
\[
{\tr}_{\tilde H}|\psi\rangle \langle \psi | =\sum_j \rho_j |\xi_j \rangle \langle \xi_j|=\rho.
\]

The possibility of purification gives rise to the important measures of distances between the states.
Namely, one defines the {\it fidelity}\index{fidelity} and the {\it fidelity distance}\index{fidelity distance}
between two states $\rho$ and $\ga$ respectively as
\[
F(\rho, \ga)=\max\{|\langle \xi|\eta \rangle|: {\tr}_{\tilde H} (|\xi\rangle \rangle \xi|)=\rho,
\quad {\tr}_{\tilde H} (|\eta\rangle \rangle \eta|)=\ga \},
\]
\[
d_F(\rho, \ga)=\min\{\| \, |\xi\rangle-|\eta \rangle \|: {\tr}_{\tilde H} (|\xi\rangle \rangle \xi|)=\rho,
\quad {\tr}_{\tilde H} (|\eta\rangle \rangle \eta|)=\ga \}.
\]

Employing the bases $|e_i^A\rangle\langle e_j^A|$ in $\TC_s(H_A)$ and
$|e_i^B\rangle\langle e_j^B|$ in $\TC_s(H_B)$ one can describe  an operator
$T:\TC_s(H_A)\to\TC_s(H_B)$ via its matrix
 \begin{equation}
\label{eqoperamatri}
T_{(j,l),(i,k)}={\tr} [|e_l^B\rangle\langle e_k^B| T(|e_i^A\rangle\langle e_j^A|)]
=\langle e_k^B | T(|e_i^A\rangle\langle e_j^A|) e_l^B \rangle,
\end{equation}
so that
\begin{equation}
\label{eqoperamatri1}
T(|e_i^A\rangle\langle e_j^A|)=\sum_{l,k} {\tr} [|e_l^B\rangle\langle e_k^B| T(|e_i^A\rangle\langle e_j^A|)] |e_l^B\rangle\langle e_k^B|
=\sum_{l,k} T_{(j,l),(i,k)} |e_l^B\rangle\langle e_k^B|.
\end{equation}

This matrix provides another representation for $T$ as an operator in $H_A\otimes H_B$ acting as
\begin{equation}
\label{eqoperamatri2}
T (e_k^A\otimes e_l^B) =\sum_{i,j} T_{(ij),(kl)} e_i^A\otimes e_j^B.
\end{equation}

Of interest are the invertible operations and the operations that preserve pure states.
As an example let us see how they look like for the qubits (for the extension to arbitrary Hilbert spaces (see \cite{Davies76})).

\begin{prop}
\label{propstrucoper}
Let $T:\TC_s(\C^2)\to\TC_s(\C^2)$ be an operation (a positive linear contraction).

(i) If $T$ is invertible and $T^{-1}$ is also an operation, then
\begin{equation}
\label{eq1propstrucoper}
T(A)=UAU^{-1}
\end{equation}
with $U$ a unitary or anti-unitary operator in $\C^2$.

(ii) If $T$ preserves pure states, then either $T$ is given by \eqref{eq1propstrucoper} or
\begin{equation}
\label{eq2propstrucoper}
T(\rho)={\tr}[\rho B]|\psi\rangle\langle \psi|
\end{equation}
with some $B\in \LC(C^2)$ and a unit vector $\psi$.
\end{prop}

\begin{proof}
(i) Any element $A$ of $\TC_s(\C^2)$ can be written as \eqref{eqrepqubitstatespace}
with real $x_j$. The operator $A$ is seen to be positive if and only if
\[
x_0\ge 0 \,\, \text{and} \,\,  x_0^2-x_1^2-x_2^2-x_3^2\ge 0.
\]
Let $T$ be given by the matrix $T_{jk}$, $j,k\in \{0,1,2,3\}$, in the coordinates
$\{x_j\}$. First of all, since both $T$ and $T^{-1}$ do not increase trace, it
follows that they both preserve trace, and hence $T_{00}=1$ and
$T_{01}=T_{02}=T_{03}=0$. Therefore $T$ can be described by the vector
$a=(T_{10}, T_{20}, T_{30})$ and the $3\times 3$-matrix $B=T_{ij}$ with $i,j\ge 1$.

Let us denote by $x$ the vectors in $\R^3$  with coordinates $x_1, x_2, x_3$.
The condition of the preservation of positivity implies that
\[
\|x\|\le 1 \implies \|a+Bx\|\le 1.
\]
Hence $B$ does not increase volume and therefore $|\det B|\le 1$. Since
the same is true for $B^{-1}$ it follows that $\det (B)=\pm 1$. Consequently
$B$ maps the unit ball $\|x\|\le 1$ bijectively and onto the unit ball
centered at $-a$. But by linearity the image of $B$ should be a symmetric
set (with each $y$ it should also contain $-y$), and consequently $a=0$.
Hence $B$ maps $\|x\|\le 1$ bijectively onto itself and hence it is a linear
 isometry and thus an orthogonal matrix. If $B\in SO(3)$, then $T$ is obtained
via the dressing with $u\in SU(2)$. If $B\in O(3)$ with $\det (O)=-1$, then
$T$ is obtained by dressing with an anti-unitary operator (see \eqref{eqantiunitdress}).

(ii) Let us prove it under the additional simplifying assumption that $T$
preserves the trace (general case just a bit more lengthy). Then, as in (i),
 we can conclude that $T_{00}=1$ and $T_{01}=T_{02}=T_{03}=0$, and $T$ can
 be described by the vector $a=(T_{10}, T_{20}, T_{30})$ and the
 $3\times 3$-matrix $B=T_{ij}$ with $i,j\ge 1$. By the requirement of the
 preservation of pure states it then follows that $T$ transform the ball
 $\|x\|\le 1$ into the ball $\|a+Bx\|\le 1$ in such a way that the boundary
 is also transformed into the boundary. It is then seen by simple topological
considerations that this is possible either when $a=0$ and $B$ is invertible
 or when $B=0$. In the first case we are back to (i), and in the second case
\[
T: x_0 I +x_1\si_1 +x_2\si_2 +x_3\si_3 \mapsto x_0 (I +a_1\si_1 +a_2\si_2 +a_3\si_3).
\]
By the requirement of the conservation of purity, $a$ is a vector of unit norm,
and we get \eqref{eq2propstrucoper}.
\end{proof}

\begin{remark} Much more complicated argument allows one to fully characterise all
positive  linear contractions $T:\TC_s(\C^2)\to\TC_s(\C^2)$. It turns out that any such operation
is a finite sum of the operators of two types (referred to as completely positive and completely
 copositive operators): $\rho \mapsto V\rho V^*$ and $\rho \mapsto V\rho^T V^*$,
where $\rho^T=\bar \rho$ is the transpose matrix to $\rho$
(that equals to its complex conjugate by self-adjointness) and $V$ some linear operators in $\C^2$,
see  \cite{WoronPos76}, or other arguments leading to the same conclusion in \cite{Stormer} and \cite{MajevPosit12}.
This result does not extend to higher dimensions.
\end{remark}

The most important class of operations constitute the so called
{\it completely positive (CP)} operations.
To define them let us notice
that for any Hilbert space $H$ the tensor product $H\otimes \C^n=H^n$ can be
represented by $n$-dimensional vectors with elements from $H$:
\[
h=(h_1, \cdots , h_n)=h_1\otimes e_1 +\cdots + h_n \otimes e_n.
\]

Moreover, the state space
\[
\TC(H\otimes \C^n) =\TC(H)\otimes \TC(\C^n)
\]
can be identified with the space of $n\times n$-matrices with elements from $\TC(H)$.
In fact, if $\rho \in \TC(H)$ and $A=(A_{jk})\in \TC(\C^n)$, then
\[
 (\rho \otimes A)(h)=\sum_j (\rho \otimes A)(h_j\otimes e_j)=
 \sum_{j,k} \rho (h_j) A_{kj}e_k,
\]
so that
\[
[(\rho \otimes A)(h)]_k=\sum_j A_{kj} \rho (h_j).
\]

Moreover, for $A=(A_{jk})\in \TC(H\otimes \C^n)$, we can write
\[
A=\sum_{j,k} A_{jk} \otimes E_{jk}
\]
with the matrix $E^{jk}$ with the elements $[E^{jk}]_{lm}=\de_j^l\de^k_m$, and therefore
\begin{equation}
\label{eqparttraceinmatr}
{\tr}_{\C^n} A =  \sum_j A_{jj}.
\end{equation}

Any linear operator $T:\TC(H_A)\to\TC(H_B)$ can be lifted to
$T_n: \TC(H_A \otimes \C^n) \to \TC(H_B\otimes \C^n)$ as
\[
T_n (X \otimes Y)=(T(X) \otimes Y).
\]
This definition means that $T_n$ acts on matrices with elements from $\TC(H_A)$
by transforming each elements by means of $T$.

Physics arguments suggest that realizable transformations of quantum state spaces
 $\TC(H)$ should be positive and remain positive after lifting to $H\otimes \C^n$.
 Thus a linear map $T:\TC(H_A)\to\TC(H_B)$ is called $n$-{\it positive}, if
 $T_n: \TC(H_A^n) \to \TC(H_B^n)$ is positive, that is
 \[
(x,T_n(a)x)=\sum_{i,j} (x_i, T(a_{ij})x_j)\ge 0,
\]
for any $x=(x_1, \cdots , x_n)\in H_B^n$, $a\in \TC(H_A^n)$ positive.
The map $T$ is {\it completely positive (CP)}\index{completely positive maps}
 if $T_n$ are positive for all $n\in N$.

From the definition of duality \eqref{eqdualoper},
it is straightforward to see that $T$ is CP if and only if $T^*$ is CP.

 One is mostly interested in trace preserving CP map, referred to as $TP-CP$
operations. Such operations are also called {\it quantum communication channels}
or {\it quantum-quantum channels} (in quantum communications) or
{\it physically realizable operations} (in quantum computing).

Since any positive operator $a\in \TC(H_A^n)$ can be written as $a=b^*b$
with some $b\in \TC(H_A^n)$ (for instance one can choose $b=b^*=\sqrt{a}$), so that
$a_{ij}=\sum_k b^*_{ki}b_{kj}$ with some $b_{ik} \in \LC(H)$,
and the operator $b^*b$ is positive for any $b\in \TC(H_A^n)$,
it follows that $T_n$ is positive if and only if
\begin{equation}
\label{eqdefcompos}
\sum_{i,j,k} (x_i, T(b^*_{ki} b_{kj})x_j)= (x_i, T[(b^*b)_{ij}]x_j)\ge 0
\end{equation}
for any $x=(x_1, \cdots, x_n)\in H_B^n$.

Choosing matrices $b$ such that $b_{ki}=\de^k_m b_i$ with some fixed
$m$ and some $b_i$ (that is a matrix with only one non-vanishing row)
 it follows from \eqref{eqdefcompos} that
\begin{equation}
\label{eqdefcompos1}
\sum_{i,j} (x_i, T(b^*_i b_j)x_j)\ge 0
\end{equation}
for any $x=(x_1, \cdots, x_n)\in H_B^n$ and $b_i \in \LC(H)$.
In other words, the matrix $T[b^*_ib_j]$ is positive definite for any
$b_1, \cdots, b_n \in \TC(H_A)$. Mappings $T$ for which this holds for
any $n$ are often referred to as the mappings of {\it positive type}
or {\it positive definite}. By linearity, \eqref{eqdefcompos1} also
implies \eqref{eqdefcompos}. Thus we arrive at important conclusion that
$T$ being of positive type is an equivalent property to being completely positive.

\begin{remark} A mapping $E: \Om\times \Om \to \TC(H)$ for any set $\Om$ is called positive definite
or of positive type if the matrix $E(a_i,a_j)$ is positive definite in $H^n$ for any $n$ and
$a=(a_1, \cdots, a_n)\in \Om^n$. Thus positive type used above refers to the mapping
$\TC(H_A)\times \TC(H_A) \to \TC(H_B)$ given by $(a,b) \mapsto a^*b$.
\end{remark}

The following result from \cite{Stinespring} gives the fundamental
{\it Stinespring representation}\index{Stinespring representation} for CP maps.

\begin{theorem}
\label{thStinespringrep}
The linear mapping $T:\TC(H_A) \to \TC(H_B)$ is completely positive, if and only if
\begin{equation}
\label{eq1thStinespringrep}
T(X)=V^*\rho(X) V
\end{equation}
with some $*$-representation $\rho$ of $\TC(H_A)$ in some finite-dimensional Hilbert space $H$
(that is $\rho$ is a linear mapping $\TC(H_A)\to \TC(H)$ such that
$\rho(a^*)=[\rho(a)]^*$, $\rho(\1)=\1$, and $\rho(ab)=\rho(a)\rho(b)$)
and a linear $V:H_B\to H$. If $T(\1)=\1$, then the operator $V$
is a (possibly partial) isometry: $V^*V=\1$.
\end{theorem}

\begin{proof}
If \eqref{eq1thStinespringrep} holds, then
\[
\sum_{i,j} (x_i, T(b^*_i b_j)x_j)
=\sum_{i,j} (Vx_i, \rho(b^*_i)\rho( b_j)Vx_j)
\]
\[
=\sum_{i,j} (\rho(b_i) Vx_i, \rho( b_j)Vx_j)
= \| \sum_i\rho(b_i) Vx_i\|^2 \ge 0.
\]

Conversely, let $T$ be CP. On the tensor product $\TC(H_A) \otimes H_B$ we can define the
Hermitian form (linear with respect to the second variable and conjugate linear with respect
to the first one) as follows:
\[
(\phi, \psi)_T= \sum_{i,j} (x_i, T(b^*_ia_j) y_j)
\]
for $\phi=\sum_i b_i\otimes x_i$, $\psi=\sum_j a_j \otimes y_j$.
By \eqref{eqdefcompos} this form is positive definite, that is
$(\phi,\phi)_T\ge 0$ for any $\phi$. Hence it satisfies the Cauchy-Schwarz inequality
$|(\phi,\psi)_T|^2\le (\phi,\phi)_T(\psi,\psi)_T$. Therefore
the null-space $N$ of this form, $N=\{\phi:(\phi,\phi)_T=0\}$, is a closed subspace
and the quotient space $H=\TC(H_A) \otimes H_B/N$ is a Hilbert space.

The natural representation $\rho :\TC(H_A) \to \TC(H)$ is obtained by projecting
the mapping $\rho': \TC(H_A) \to \TC( \TC(H_A) \otimes H_B)$ defined by
\[
\rho'(X):   \sum_i b_i\otimes x_i \mapsto \sum_i Xb_i\otimes x_i,
\]
to $H$. The mapping $x\to \1 \otimes x$ induces the linear operator $V:H_B\to H$.
Equation
\[
(\rho(X) Vx, Vx)_T=(X\otimes x, \1\otimes x)_T=(T(X)x,x)
\]
implies \eqref{eq1thStinespringrep}.
\end{proof}

Theorem \ref{thStinespringrep} implies the following corollary. The linear mapping
 $T:\TC(H_A)\to\TC(H_B)$ is CP if and only it is $\min(n,m)$-positive, where $n$ and $m$
 are the dimensions of $H_A$ and $H_B$. In fact, in the proof above only $n$-positivity
 was used. If $n>m$, we can turn to the adjoint mapping $T^*$, where $m$-positivity
 would suffice to get the Stinespring representation.

Using the theory of representations, one can make formula \eqref{eq1thStinespringrep}
even more concrete. Namely, it is known (see e.g. Section 22 of \cite{Naimark}) that
any representation $\rho$ of $\TC(H_A)$ in some finite-dimensional Hilbert space $H$
is equivalent (up to a trivial representation) to the direct sum of a finite number
 of identical representations. That is, $H=H_0+\sum_{k=1}^K H_k$ (orthogonal sum of subspaces) and
\[
\rho(X)=\sum_{k=1}^K U_k^* X U_k
\]
with $U_k$ isometric bijections $H_k\to H$. Substituting this formula in
  \eqref{eq1thStinespringrep} one obtains the following fundamental
 {\it Kraus} or  {\it Choi-Kraus representation}\index{Choi-Kraus representation} for CP maps.

 \begin{theorem}
\label{thChoiKrausrep}
The linear mapping $T:\TC(H_A) \to \TC(H_B)$ is completely positive, if and only if
\begin{equation}
\label{eq1ChoiKrausrep}
T(X)=\sum_{k=1}^K V_k X V_k^*
\end{equation}
with $K\le nm$ ($n$ and $m$
 are the dimensions of $H_A$ and $H_B$)
 and some linear operators $V_k$ in $H_A$. If $T$ is trace preserving, then $\sum_k V_k^* V_k=\1$.
\end{theorem}

Notice that the last statement is obtained by using the fact that trace preservation of $T$ means that
$T^*$ preserves the identity operator.

As the composition of the operators of type \eqref{eq1ChoiKrausrep} is clearly of the same type,
it follows that the composition of any two CP maps is again CP.

Yet another representation of CP maps in terms of
partial traces (also referred to sometimes as the
{\it Stinespring representation}\index{Stinespring representation})
is of great importance for physical interpretation and realization of
 these maps (see e.g. \cite{Hayashi} or \cite{Kraus}). Let us present it for coinciding $H_A$ and $H_B$
  (see more general versions in \cite{Hayashi} or \cite{Kraus}).
 \begin{theorem}
\label{thStinespr2rep}
 For any TP-CP map  $T:\TC(H) \to \TC(H)$ there exists a Hilbert space $\tilde H$,
 and (i) a partial isometry $F: H \to H\otimes \tilde H$ such that
\begin{equation}
\label{eq1thStinespr2rep}
T(\rho)={\tr}_{\tilde H} [F \rho F^*],
\end{equation}
and (ii) a unitary map $U: H \otimes \tilde H \to H\otimes \tilde H$ and a state $\om \in \TC(\tilde H)$ such that
\begin{equation}
\label{eq2thStinespr2rep}
T(\rho)={\tr}_{\tilde H} [U (\rho \otimes \om)  U^*].
\end{equation}
\end{theorem}

\begin{proof}
(i) Let $\tilde H=\C^K$ with $K$ the number of terms in \eqref{eq1ChoiKrausrep},
and let $F:H\to H\otimes \tilde H$ is defined by the formula
\[
F=\sum_k V_k \otimes |k\rangle : |\xi\rangle \mapsto \sum_k V_k |\xi \rangle \otimes |k\rangle,
\]
where $|k\rangle$ are the basis vectors in $\tilde H$. It follows that
\[
F^*(x\otimes y)=\sum_k V_k^* x \langle k| y \rangle,
\]
which can be denoted $\sum_k V_k^* \otimes \langle k|$ by identifying
$H$ with $H\otimes \C$. Hence
\[
F^*F   =\sum_{k,j} ( V_k^* \otimes \langle k|)( V_j \otimes |j\rangle)
=\sum_{k,j} V_k^* V_j \langle k|j\rangle =\sum_k V_k^* V_k =\1,
\]
that is, $F$ is a partial isometry.
Finally, by \eqref{eqdefparrtrace2},
\[
\sum_k V_k \rho V_k^* =\sum_{k,j} {\tr}_{\tilde H}(V_k \rho V_j^* \otimes |k\rangle \langle j|)
={\tr}_{\tilde H}(F \rho F^*).
\]

(ii) Let $\om=|1\rangle \langle 1|$. Then $(\rho \otimes \om)_{jk}=\rho \de^j_k$ for any $\rho \in \TC(H)$.
Hence
\[
[U(\rho \otimes \om) U^*]_{jk}=U_{k1} \rho U^*_{j1}
\]
for any operator $U=(U_{jk}) \in \TC(H \otimes \tilde H)$. Therefore, by \eqref{eqparttraceinmatr},
in order to get \eqref{eq2thStinespr2rep} we need
\[
\sum_k U_{k1}\rho U^*_{k1}=\sum_k V_k X V_k^*.
\]
Thus it is sufficient to have $U_{k1}=V_k$ for all $k$. Clearly a unitary operator $U$ exists
satisfying this condition, because $\sum_k V_k^* V_k=\1$.
\end{proof}

The physical meaning of \eqref{eq1thStinespr2rep} is as follows. It means that the $TP-CP$
transformations are exactly the transformations obtained from pure, that is unitary, transformations
performed on a given system combined with another ancillary system, referred to as a reservoir,
or environment, or just {\it ancilla}\index{ancilla}, and projected on the states of a given system.

 Yet another characterization of CP-TP map can be given in terms of its
matrix  \eqref{eqoperamatri}.

 \begin{theorem}
\label{thposdefmatforCP}
$T$ is completely positive iff its matrix
\eqref{eqoperamatri} is positive, as the matrix of an operator in  $H_A\otimes H_B$, that is,
\begin{equation}
\label{eqposdefmatforCP}
\sum_{j,l,i,k} \bar x_{jl}T_{(j,l),(i,k)} x_{ik} \ge 0
\end{equation}
for any vector $(x_{jl})\in H_A\otimes H_B$.
\end{theorem}

\begin{proof}
(i)  Let $T$ is completely positive.
By Theorem \ref{thChoiKrausrep} and convexity, it is sufficient to show that for $T(\rho)=V\rho V^*$, the matrix
\begin{equation}
\label{eq1posdefmatforCP}
T_{(j,l),(i,k)}=\langle e_k^B |V (|e_i^A\rangle\langle e_j^A|)V^* e_l^B \rangle
=\langle e_k^B |V e_i^A\rangle\langle V e_j^A|) e_l^B \rangle=y_{ik}\bar y_{jl}
\end{equation}
is positive, where $y_{ik}= \langle e_k^B| Ve_i^A\rangle$. But
\[
\sum_{j,l,i,k} \bar x_{jl}T_{(j,l),(i,k)} x_{ik}
=\sum_{ik} x_{ik} y_{ik}  \sum_{jl} \bar x_{jl} \bar y_{jl}
=\left| \sum_{ik} x_{ik} y_{ik} \right|^2 \ge 0.
\]

(ii) Again by convexity, to show that any positive matrix corresponds to CP map, it is sufficient
to show this for the extreme points of the set of positive matrices that have the form $T=|Y\rangle\langle Y|$
with some $Y=\sum_{ij}y_{ij} e_i^A\otimes e_j^B\in H_A\otimes H_B$. These operators act in $H_A\otimes H_B$ as
\[
T(e_k^A\otimes e_l^B)=\bar y_{ij} y_{kl} e_i^A\otimes e_j^B,
\]
and thus have the matrix $T_{ij,kl}=\bar y_{ij} y_{kl}$. By \eqref{eq1posdefmatforCP}, this is the matrix
of the operator  $T(\rho)=V\rho V^*$  with $y_{ik}= \langle e_k^B| Ve_i^A\rangle$.
\end{proof}

As shown in Proposition \ref{propstrucoper}, positive TP maps in $\TC(\C^2)$ can be described
by a vector $a$ and an operator $B$ in $\R^3$. Complete positivity can be fully
characterised in terms of certain inequalities involving $a$ and singular values
 of $T$. This characterization is however rather nontrivial, see \cite{ResSz02}.
An interesting point to note is that the operations given by \eqref{eq1propstrucoper}
with anti-unitary $U$ are positive, but not CP.

Let us point out some basic examples of CP maps.
Pinching map \eqref{eqantiunitdresscondexp} is a CP map, because it is explicitly
defined in the Kraus representation. Given $\rho$, a state on $\tilde H$,
the mapping $\si \mapsto \si\otimes \rho$ from $\TC(H) \to \TC(\tilde H)$
 is clearly CP, and consequently, the partial trace \eqref{eqdefparrtrace}
 is also CP by the duality. In qubits the mappings
 \begin{equation}
\label{eqPauliChan}
T(\rho)=p_0 \rho +\sum_{j=1}^3 \si_j \rho \si_j,
\end{equation}
with a probability distribution $p_0,p_1,p_2,p_3$, are called the
{\it Pauli channels}. They are CP due to the Kraus representation.
In the $d$-dimensional case $\C^d$ with the basis $e_0, \cdots, e_{d-1}$
one introduces the operators $X$ and $Z$ by their actions $Xe_j=e_{j-1}$
for $j>0$ and $Xe_0=e_{d-1}$, and $Ze_j=e^{-2\pi i/d} e_j$ for all $j$.
The {\it generalized Pauli channel}\index{Pauli channel}, defined by the formula
 \begin{equation}
\label{eqPauliChangen}
T(\rho)=\sum_{k,j=0}^{d-1} p_{kj} (X^kZ^j)^* \rho (X^kZ^j),
\end{equation}
where $\{p_{ij}\}$ is a probability law on $\{0, \cdots, d-1\}^2$, is also a CP map.

\section{Elements of the general theory of quantum games}
\label{secgenthequaga}

A general static (simultaneous) game of $N$ players is a triple
$(N, S=S_1\times \cdots \times S_N, \Pi=(\Pi_1, \cdots, \Pi_N)$,
where $S_j$ is the strategy space of $j$th player and $P_j: S\to \R$
is the payoff of the $j$th player. General quantum games can be fit
into this scheme. Namely, a static simultaneous {\it quantum game}\index{quantum game}
of $N$ players with finite-dimensional strategies (and separated actions) can be described by
$N$ finite-dimensional Hilbert spaces $H_1, \cdots, H_N$, an initial state
$\rho$ on $H=H_1\otimes \cdots \otimes H_N$, a POVM $\{M_{\om}\}$, $\om \in \Om$, on $H$
(see \eqref{eqdefPOVM}, \eqref{eqdefPOVM1}) with the set of outcomes $\Om$, the payoff
functions $f=(f_1, \cdots , f_N)$, $f_j:\Om \to \R$, with $f_j(\om)$ being the payoff of $j$th player
for the outcome $\om$, and the choice of strategic spaces $S_j$ for each player,
where $S_j$ is a closed subset of the set of all CP-TC mappings in $\TC(H_j)$.
For a choice (or profile) of strategies $(s_1, \cdots, s_N)\in S$, the final state of the game
is assumed to be $(s_1\otimes \cdots \otimes s_N)\rho$ with the possible outcome $\om$ measured by the POVM $\{M_{\om}\}$
being
\[
{\tr} [M_{\om} (s_1\otimes \cdots \otimes s_N)\rho],
\]
so that the final payoffs can be calculated by the formula
 \begin{equation}
\label{eqgenquagapa}
\Pi_j(s_1, \cdots, s_N)=\sum_{\om\in \Om} f_j(\om) {\tr} [M_{\om} (s_1\otimes \cdots \otimes s_N)\rho].
\end{equation}

Notice that the introduction of POVM generalises both MW and EWL protocols.

Let us say that the quantum game is played with the {\it full strategic spaces}
\index{quantum game!full strategic spaces} if each $S_j$ coincide
with the whole set  CP-TC mappings in $\TC(H_j)$ and the quantum game is played with
 {\it full unitary strategic space} (sometimes referred to in this context as pure strategies)
 \index{quantum game!pure or unitary strategic space} if each $S_j$ arises from the set of all unitary
operators in $H_j$.

The following quantum version of the Nash theorem is a straightforward extension
of its classical counterpart.

 \begin{theorem}
\label{thquantumNash}
Any quantum game played with the full strategic spaces has at least one Nash equilibrium.
\end{theorem}

\begin{proof}
Since the strategic spaces $S_j$ are compact convex sets (as closed subsets of trace preserving positive maps, see
Lemma \ref{lempositivearecompact}) and the payoff function \eqref{eqgenquagapa} is linear on each $S_j$,
the proof is exactly the same as the classical version, or otherwise stated, the claim is a particular case
of the general Glicksberg theorem (see \cite{KolMal} or other books on game theory).
\end{proof}

\begin{remark} Possibly the first precise formulation of this general result appeared in \cite{LeeJo03},
but it was mentioned in particular forms in previous publications.
\end{remark}

\section{Further links and examples}
\label{secfinconq}

Concluding our introduction to quantum games let us note that the literature on this subject
is already quite immense. Further general insights and extensive bibliography can be obtained
from various review papers that include \cite{GuoZhang08}, \cite{KhanReview18},
\cite{GrabbeGameComput}. There one can find also references to the big chunk of work devoted
to building various quantized versions of all standard examples of classical games (various
social dilemmas, etc, like the quantum versions of Monty Hall problem in \cite{FitneyMontyHall02}
and \cite{Chuan01}, and of the Trucker Game in \cite{DuLi02}, see also  "clever Alice'' and
"stupid Alice'' from \cite{Grib}).
Let us indicate some trends of research which were not even touched upon in our presentation.
These trends include the repeated or iterated quantum games initiated in \cite{KayRepGame01}, the analysis of the links
of quantum games with the Bayesian games of incomplete information (see \cite{CheonIqBqyes}), and an interesting activity on the
expressing (interpreting) in game-theoretic term the fundamental properties of quantum nonlocality and its optimal quantitative
characteristics, see e.g. \cite{FlitneyBell09} and \cite{IqAbbBell}, linking the theory of games with the fundamental problems of
quantum communication and teleportation. Another natural development, which is seemingly not explored so far, would be
the theory of dynamic games built on the basis of quantum filtering, as was initiated in 
\cite{Bel87} and \cite{Kol92} for quantum control.
Finally, the initial paper \cite{MeyerD99} being motivated by problems in quantum
computation and cryptology, this link is of great importance, see e.g. \cite{GrabbeGameComput}.
As the simplest example illustrating this link let us describe briefly
the well known 'Guess the number' game.

Recall that $n$-qubit systems can be described by the Hilbert space $H^{\otimes n}$, which is the tensor product of
$n$ two-dimensional spaces $H=\C^2$. The space $H^{\otimes n}$ has dimension $2^n$ and its natural basis can be
represented by $2^n$ vectors of the form $|x\rangle$, where
\[
x=|x_{n-1} \cdots x_0\rangle =|x_{n-1} \rangle \otimes \cdots \otimes |x_0\rangle
\]
is a string of $n$ symbols $x_j$
with values $0$ or $1$ representing the binary expansion of the corresponding
integer $x$. Let $x\cdot y$ denotes the modulo $2$ scalar product of these expansions:
\[
x\cdot y =(x_{n-1}y_{n-1} + \cdots + x_0y_0) (mod \,\, 2).
\]
The {\it Hadamard operator} or {\it Hadamard-Walsh operator}\index{Hadamard-Walsh operator} on a qubit
is the transformation of $\C^2$ given by the matrix
\[
W=\frac{1}{\sqrt 2}\left(\begin{aligned} & 1 \quad \quad 1 \\ & 1 \quad -1 \end{aligned} \right),
\]
or equivalently by its action on the standard basis:
\[
W|0\rangle =\frac{1}{\sqrt 2}(|0\rangle+|1\rangle),
\quad W|1\rangle =\frac{1}{\sqrt 2}(|0\rangle-|1\rangle).
\]
The {\it Hadamard-Walsh operator}\index{Hadamard-Walsh operator} on $n$-qubit system
is the tensor product $W^{\otimes n}$ acting as
\[
W^{\otimes n} \left(|x_{n-1} \rangle \otimes \cdots \otimes |x_0\rangle\right)
=W|x_{n-1} \rangle \otimes \cdots \otimes W|x_0\rangle.
\]
Clearly $W$ is a unitary operator such that $[W^{\otimes n}]^2=\1$. Direct computation shows that
 \begin{equation}
\label{eqWalshHad}
W^{\otimes n} |0\cdots 0 \rangle=2^{-n/2} \sum_{x=0}^{2^{n-1}-1} |x\rangle,
\end{equation}
the r.h.s. being the uniform mixtures of all basis states of $H^{\otimes n}$, and generally
 \begin{equation}
\label{eqWalshHad1}
W^{\otimes n} |y\rangle=\sum_{x=0}^{2^{n-1}-1} (-1)^{x\cdot y}|x\rangle.
\end{equation}

The {\it Bernstein-Vazirani oracle}\index{Bernstein-Vazirani oracle}
 with a parameter $a\in H$ is the transformation of $H$ defined by the following
action on the basic vectors $T_a |x\rangle = (-1)^{a\cdot x} |x\rangle$.
The {\it 'Guess the number' game}\index{Guess the number game} we are talking here is the game between Alice and Bob, where
Alice chooses a number $a$ and Bob has to guess it by asking the result of the action of the oracle
on some vectors. How many question Bob has to ask to guess $a$? Classically, when he can use only the basis vectors
$|x\rangle$, he needs effectively to get the results for all $2^n$ vectors thus asking $2^n$ questions.
Remarkably enough, using the full space $H^{\otimes n}$ he can find the answer just with one question. Namely,
Bob prepares the initial state
\[
|\psi\rangle=2^{-n/2} \sum_{x=0}^{2^{n-1}-1} |x\rangle,
\]
and asks Alice to give him the result of $T_a |\psi\rangle$, that is,
\[
T_a |\psi\rangle=2^{-n/2} \sum_{x=0}^{2^{n-1}-1} (-1)^{a\cdot x}|x\rangle =W^{\otimes n} |a\rangle.
\]
It remains for Bob to apply another $W^{\otimes n}$ to get the required number
$a =W^{\otimes n} T_a |\psi\rangle  =W^{\otimes n} W^{\otimes n} |a\rangle$.

\end{document}